%% file: main.tex
\documentclass{article}
\usepackage{graphicx} 

\usepackage[utf8]{inputenc}
\usepackage[T1]{fontenc}
\usepackage[english]{babel}
\usepackage{geometry}
\usepackage{wrapfig}

\usepackage{authblk}

\usepackage{amsfonts}
\usepackage{amsmath}
\usepackage{amssymb}
\usepackage{amstext}
\usepackage{amsthm}

\usepackage[shortlabels]{enumitem}
\usepackage{float}
\usepackage{graphicx}
\usepackage{hyperref}
\usepackage{listings}
\usepackage{mathrsfs}
\usepackage{mathtools}
\usepackage{multicol}
\usepackage{wasysym}
\usepackage{subcaption}
\usepackage{amstext}
\usepackage{subfiles}

\newtheorem{thm}{Theorem}[section]
\newtheorem{conj}[thm]{Conjecture}
\newtheorem{cor}[thm]{Corollary}
\newtheorem{lem}[thm]{Lemma}

\newtheorem{obs}[thm]{Observation}
\newcounter{claim}[thm]

\newenvironment{property}[1][]%
{\vspace{2ex} \noindent \textbf{Property~($\star$)} #1}{\vspace{2ex}}
\newcommand{\MP}{Property~{\rm($\star$)} }
\newcommand{\MPdot}{Property~{\rm($\star$)}}

\newenvironment{propertye}[1][]%
{\vspace{2ex} \noindent \textbf{Property~($\circ$)} #1}{\vspace{2ex}}
\newcommand{\MPe}{Property~{\rm($\circ$)} }
\newcommand{\MPedot}{Property~{\rm($\circ$)}}

\newenvironment{claim}[1][]%
{ \refstepcounter{claim}\vspace{1ex} {\it Claim \it \arabic{claim}. {#1}{}}\it}{\vspace{2ex}}
\newenvironment{proofclaim}[1][]%
{\noindent {}{#1}}{This proves Claim~\arabic{claim}.\vspace{2ex}}

\theoremstyle{definition}

\usepackage[dvipsnames]{xcolor}

\newcommand{\ER}[1]{{\color{purple}{\bf ER:} #1}}

\usepackage{multicol}
\usepackage[textsize=small]{todonotes}

\renewcommand{\phi}{\varphi}

\DeclareMathOperator{\codeg}{codeg}
\DeclareMathOperator{\mcodeg}{\Delta_2}
\DeclareMathOperator{\acodeg}{\overline{\codeg}}
\DeclareMathOperator{\macodeg}{\overline{\Delta_2}}

\usepackage{thm-restate}

\newcommand{\say}[1]{``#1''} 

\usepackage{soul}
\usepackage{comment}

\title{Vu's conjecture holds for claw-free graphs}
\author[1]{Linda Cook} 
\author[2]{Ross J. Kang} 
\author[3]{Eileen Robinson} 
\author[2]{Gabri\"elle Zwaneveld}
\date{5 June 2026}

\affil[1]{\small Mathematical Institute, Utrecht University, Netherlands.}
\affil[2]{\small Korteweg–de Vries Institute for Mathematics, University of Amsterdam, Netherlands.}
\affil[3]{\small Universit\'e libre de Bruxelles, Belgium.} 

\begin{document}
 \maketitle

\begin{abstract}
Given a graph $G$, let $\Delta_2(G)$ denote the maximum number of neighbors any two distinct vertices of $G$ have in common.
Vu (2002) proposed that, provided $\Delta_2(G)$ is not too small as a proportion of the maximum degree $\Delta(G)$ of $G$, the chromatic number of $G$ should never be too much larger than $\Delta_2(G)$.
We make a first approach towards Vu's conjecture from a structural graph theoretic point of view.
We prove that, in the case where $G$ is claw-free, indeed the chromatic number of $G$ is at most $\Delta_2(G)+3$.
This is tight, as our bound is met with equality for the line graph of the Petersen graph. 
Moreover, we can prove this in terms of the more specific parameter that bounds the maximum number of neighbors any two endpoints of some edge of $G$ have in common.
Our result may be viewed as a generalization of the classic bound of Vizing (1964) for edge-coloring.


\end{abstract}

\section{Introduction}\label{sec:intro}

The {\em codegree}  $\codeg(v,w)$ of two distinct vertices $u$ and $v$ in a graph $G$ is the number of common neighbors of $u$ and $v$ in $G$. The {\em maximum codegree} $\mcodeg(G)$ of $G$ is defined as $\mcodeg(G)= \max_{u,v \in V(G)} \codeg(u,v)$. We denote the maximum degree of $G$ by $\Delta(G)$.

Our starting point is the following conjecture for the chromatic number $\chi(G)$ of $G$ in terms of its maximum codegree $\mcodeg(G)$. This is essentially a conjecture of Vu~\cite{Vu02} from 2002; he actually suggested it for the stronger {\em list} chromatic number $\chi_\ell(G)$ of $G$.

\begin{conj}\label{conj:Vu}
Fix $\varepsilon_1,\varepsilon_2 > 0$.
If $\mcodeg(G) \ge \varepsilon_1 \Delta(G)$, then $\chi(G) \le \mcodeg(G)+\varepsilon_2\mcodeg(G)$, provided $\Delta(G)$ is sufficiently large.
\end{conj}
\noindent
The lower bound hypothesis for $\mcodeg(G)$ could potentially be weakened, but not to any bound that is $o(\Delta(G)/\log\Delta(G))$, due to random regular graphs (see~\cite{CJMS23+}).
We have no reason yet to rule out an absolute constant, possibly even $3$, in place of the error term $\varepsilon_2\mcodeg(G)$ in the bound.
To that end, below in Conjecture~\ref{conj:Vusharper} we propose a stronger form of this conjecture.
Despite its essential nature and its interesting implications for other problems such as the List Coloring Conjecture (see the discussion in~\cite{Vu02}), there was no movement on Vu's conjecture until recent partial progress employing probabilistic methodology~\cite{HJK22,KKO24,BDMW25+}. %
In general, Conjecture~\ref{conj:Vu} remains wide open.

Distinct from earlier efforts, we embark on a {\em structural} graph theoretic investigation of Conjecture~\ref{conj:Vu}. 
To take our viewpoint, consider the case where $G$ is the line graph $L(G_0)$ of some (simple) graph $G_0$.
By a straightforward case analysis, it is not difficult to see that the maximum codegree of $G$ satisfies $\mcodeg(G) \in \{\Delta(G_0)-2,\Delta(G_0)-1\}$ provided $\Delta(G_0) \ge 3$. (Which of the two depends on whether a vertex in $G_0$ of maximum degree is in a triangle.)
Recall by a classic theorem on edge-coloring due to Vizing~\cite{vizing1964estimate} that for any (simple) graph $G_0$, the line graph of $G_0$ satisfies $\chi(L(G_0)) \leq \Delta(G_0)+1$. Hence, if $G$ is a line graph, then $\chi(G) \leq \mcodeg(G)+3$ as  $\mcodeg(G) \geq \Delta(G_0)-2$.  
Thus Conjecture~\ref{conj:Vu} holds in the restricted case of $G$ being the line graph of some (simple) graph.
Note that this bound is tight for the line graph of the Petersen graph, which has maximum codegree $1$ and chromatic number $4$; see Figure~\ref{fig:petersenline}. (In fact, it is tight for the line graph of any triangle-free graph of class II.)
Note also that a result of Kahn~\cite{Kah96b}, a {\em tour de force} of the Nibble Method, implies that the list chromatic number of any line graph $G$ satisfies $\chi_\ell(G) = (1+o(1)) \mcodeg(G)$ as $\mcodeg(G)\to\infty$. 
Vu noted this as inspiration for the stronger list coloring version of Conjecture~\ref{conj:Vu}.

A {\em claw} is a copy of the star $K_{1,3}$ as an induced subgraph and we say that $G$ is {\em claw-free} if it contains no claw.
Claw-freeness is equivalent to the condition that no neighborhood contain an independent set of size three.
Every line graph is claw-free, a nontrivially strict inclusion. Before setting some wider context, here is our main result.

\begin{restatable}{thm}{maincf}
\label{thm:maincf}
If $G$ is claw-free, then $\chi(G) \le \mcodeg(G)+3$. 
\end{restatable}
\noindent
This confirms Conjecture~\ref{conj:Vu} in the restricted case of $G$ being a claw-free graph.
To the best of our knowledge, this is the first specific confirmation of Conjecture~\ref{conj:Vu} within a non-trivial hereditary graph class. 
We remark that any $\chi$-binding function for a given graph class of at most $(1+o(1))\omega$ also implies the conjecture for that class; we discuss more context for this in the conclusion.
One may interpret Theorem~\ref{thm:maincf} as nearly a direct generalization of the bound in Vizing's theorem.
(Analogously, Conjecture~\ref{conj:Vu} can be seen as, prospectively and approximately, an even broader extension of Vizing's.)
Just above, we noted how this bound is met for the line graph of, among others, the Petersen graph.

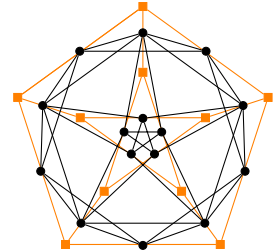
\begin{wrapfigure}{R}{0.25\textwidth}
    \centering
    \input{figures/PetersenLinegraph}
    \caption{The Petersen graph (in orange) and its line graph (in black).}
    \label{fig:petersenline}
\end{wrapfigure}

Already several efforts in structural graph theory have been devoted to extending edge-coloring results to coloring results for claw-free graphs~\cite{ChudFrad08,chudnovsky2007coloring,claws-vi-coloring,EGM14,GrMa04,king-reed-asymptotics}. These usually gave bounds as a function of the clique number. When $G$ is the line graph of $G_0$, the clique number $\omega(G)$ of $G$ coincides with the maximum degree $\Delta(G_0)$ of $G_0$, provided $\Delta(G_0)\ge 3$; and so Vizing's theorem translates into a bound of the form $\chi(G) \le \omega(G)+1$ in this case. Gy\'arf\'as~\cite{Gya87} noted how $\chi(G) \le \omega(G)^2$ for any claw-free $G$. However, there are claw-free graphs $G$ of arbitrarily large clique number $\omega(G)$ for which $\chi(G) = \Omega(\omega(G)^2/\log \omega(G))$: one can take $G$ as the complement of any construction certifying the lower bound $R(3,k) =\Omega(k^2/\log k)$ as $k\to\infty$ for the off-diagonal Ramsey numbers $R(3,k)$, e.g.~\cite{Kim95}. Thus, the direct extension to claw-free graphs of the bound in Vizing's theorem fails if cast in terms of the clique number, in contrast with our Theorem~\ref{thm:maincf}.

We briefly discuss some broader rationale for focusing on claw-free graphs in particular.
Almost half a century ago independently Minty~\cite{Min80} and Sbihi~\cite{Sbi80} demonstrated a polynomial-time algorithm to compute the independence number of a given claw-free graph, extending the analogous result for the matching number due to Edmonds~\cite{Edm65}.
Over the decades, there have been further important attempts spanning combinatorial optimization and graph theory to extend results from the much more structured class of line graphs to the wider class of claw-free graphs.
Perhaps most notably are the efforts to better understand the so-called stable set polytope of claw-free graphs, originating in problems of Ben Rebea (see~\cite{ChSb88,Ori03}) and of Gr\"otschel, Lov\'asz and Schrijver~\cite{GLS88}.
This motivated efforts towards the better structural understanding of claw-free graphs~\cite{ChudSeym05,EOSV08}.
One may view our work within this general line and, indeed, it relies heavily on claw-free structure theory.

We say that $G$ is a {\em quasi-line} graph if every neighborhood of a vertex in $G$ admits a partition into two parts so that each part induces a clique of $G$.
Every line graph is quasi-line and every quasi-line graph is claw-free, both non-trivially strict inclusions.
This is an important intermediate class in claw-free structure.
As one of the milestones towards Theorem~\ref{thm:maincf}, we highlight the following result, which we prove in Section~\ref{sec:quasiline}. 
\begin{restatable}{thm}{mainquasiline} \label{thm:mainquasiline} If $G$ is a quasi-line graph, then $\chi(G) \le \mcodeg(G)+3$.
\end{restatable}
\noindent
Note for comparison that $\chi(G) \le 1.5\omega(G)$ for quasi-line $G$ and there are examples attaining the bound~\cite{chudnovsky2007coloring}. So the direct extension even to quasi-line graphs of the bound in Vizing's theorem fails in terms of the clique number.

Given a graph $G$, let us denote by $\Delta_e(G)$ the maximum size of a collection of triangles in $G$ that all contain some common edge. In other words, $\Delta_e(G)$ is the maximum codegree taken among the pairs of endpoints of some edge in $G$. Note that $\Delta_e(G) \le \mcodeg(G)$ always. With some adaptations in the proof, Theorem~\ref{thm:maincf} can be strengthened to hold with $\Delta_e$ instead of $\mcodeg$.

\begin{restatable}{thm}{maincfedge} \label{thm:maincfedge}
    If $G$ is claw-free, then $\chi(G) \le \Delta_e(G)+3$.
\end{restatable}

\paragraph{Outline of the paper.} In Section~\ref{sec:prelim}, we state notation and some preliminary results. Next, in Section~\ref{sec:defuzz}, we show that if our main results hold for all `non-fuzzy' claw-free graphs then it also holds for all `fuzzy' claw-free graphs, implying that in the rest of the paper we only need to deal with the `non-fuzzy' graphs. Section~\ref{sec:quasiline} is dedicated to the proof of the main result on quasi-line graphs, namely Theorem~\ref{thm:mainquasiline}. Then Section~\ref{sec:cf} extends this result to claw-free graphs, completing the proof of Theorem~\ref{thm:maincf}. In Section~\ref{sec:adaptation}, we specify the adaptations required in our proof to strengthen the result to Theorem~\ref{thm:maincfedge}. We end with some thoughts for future research.

\section{Notation and preliminaries}\label{sec:prelim}

Although it is mostly standard, for the convenience of the reader, we collect much of our notation here as well as some basic preparatory observations. We also state the structure theorem we work from and outline our strategy in the proof of Theorem~\ref{thm:maincf}.

All graphs $G$ are finite and undirected, with $V(G)$ denoting the set of vertices and $E(G)$ denoting the set of edges of $G$. 
A graph $H$ is a \emph{subgraph} of $G$ if $V(H)\subseteq V(G)$ and $E(H)\subseteq E(G)$. If $H$ is a subgraph of $G$ and $e\in E(H)$ if and only if $e \in E(G)$, then $H$ is an \emph{induced subgraph} of $G$. All (sub)graphs in this paper (including special ones such as cliques and independent sets) are assumed to be non-empty and thus contain at least one vertex. 
We call a graph class $\mathcal{G}$ \emph{hereditary} if it is closed under taking induced subgraphs. 
The \emph{complement} of a simple graph $G$ is the graph $\overline{G}$ with $V(\overline{G})=V(G)$ and $e\in E(\overline{G})$ if and only if $e\notin E(G)$. Taking the complement is a self-inverse property, so $\overline{\overline{G}}=G$.
For a given graph $H$, a graph $G$ is said to be \emph{$H$-free} if it does not contain $H$ as an induced subgraph. For $\cal H$ a set of graphs, $G$ is \emph{$\cal H$-free} if $G$ is $H$-free for all $H\in \cal H$.

For $Y \subseteq V (G)$ and $x \in V (G) \setminus Y$, we say that $x$ is \textit{complete} (resp.~\textit{anticomplete}) to $Y$ if $x$ is adjacent (resp.~non-adjacent) to every member of $Y$. 
If $x$ is neither complete nor anticomplete to $Y$, we say $x$ is \emph{mixed} on $Y$.
For disjoint $X, Y \subseteq V (G)$, we say that $X$ is \textit{complete} (resp.~\textit{anticomplete}) to $Y$ if every vertex of $X$ is adjacent (resp.~non-adjacent) to every vertex of $Y$.
The \emph{open} (resp.~\emph{closed}) \emph{neighborhood} of a vertex $v\in V(G)$ is the set $N(v)=\{u\in V(G) \, \vert \, uv\in E(G), u\neq v\}$ (resp.~$N[v]=N(v)\cup\{v\}$). The \emph{degree} of a vertex $v\in V(G)$, denoted $\deg(v)$, is the size of its open neighborhood. The \emph{maximum degree} of a graph $G$ is $\Delta(G)=\max_{v\in V(G)}\deg(v)$. 
The \emph{codegree} of two distinct vertices $u,v\in V(G)$, denoted $\codeg^G(u,v)$ or $\codeg(u,v)$ when the graph used is clear, is the size of their common neighborhood: $\codeg(u,v)=|N(u)\cap N(v)|$. The \emph{maximum codegree} of a graph $G$ is $\mcodeg (G)= \max_{u,v \in V(G)} \codeg (u,v)$. 
Clearly, $\mcodeg$ is non-increasing under the operation of taking subgraphs.
The \emph{anti-codegree of $v,w\in V(G)$}, denoted $\acodeg^G(v,w)$ or $\acodeg(v,w)$ when the graph used is clear, is the number of vertices $x \in V(G)$ such that $x \neq v$ and $x \neq w$, satisfying that $x$ is not adjacent to $v$ and also not to $w$, i.e.~$\acodeg(v,w)=|V(G)\setminus (N(v)\cup N(w)\cup\{v,w\})|$. Hence, by definition of the complement graph, we obtain $\acodeg^{G}(u,v) = \codeg^{\overline{G}}(u,v)$. 
The \emph{maximum anti-codegree} of a graph $G$ is $\macodeg(G) = \max_{v,w \in V(G)} \acodeg(v,w)$. It satisfies $\macodeg(G) = \mcodeg(\overline{G})$.

A \emph{clique} is a set of pairwise adjacent vertices. A clique of size $3$ is called a triangle. A clique of $\ell$ elements is denoted by $K_\ell$.
The \emph{clique number} of a graph $G$, denoted $\omega(G)$, is the size of a largest clique in $G$. If vertices $u$ and $v$ belong to a clique of size $\omega(G)$, then all other vertices in this clique are common neighbors of $u$ and $v$. Hence, we have the following inequality between the maximum codegree and the clique number:
\[\mcodeg (G) \geq \codeg(u,v) \geq \omega(G)-2.\] 
An \emph{independent set} is a set of vertices such that no pair of vertices is adjacent. An independent set of size $3$ is called a triad. The \emph{independence number} of a graph $G$, denoted by $\alpha(G)$, is the size of a largest independent set in $G$.
A \emph{matching} in $G$ is a set of pairwise disjoint edges. The {\em matching number} of a graph $G$, denoted $m(G)$, is the size of a maximum matching in $G$.
A \emph{(vertex-)coloring} of $G$ is a function $f: V(G) \rightarrow \mathbb{Z}^+$ where the image elements are referred to as \emph{colors}. 
A coloring is called \emph{proper} if any two adjacent vertices are assigned distinct colors; we often omit this modifier when the context is clear.
The \emph{chromatic number} of a graph $G$, denoted $\chi(G)$, is the least number of colors in a proper coloring of $G$. Since the chromatic number of a graph is the maximum of the chromatic number of each of its components, we need only consider connected graphs. 
For any vertex $v$, let $L(v)$ be a list of colors for $v$. If for any choice of lists of size $k$, there exists a proper coloring using for each $v$ a color from $L(v)$, then $G$ is $k$-list-colorable. The list chromatic number is defined as $\chi_\ell(G)= \min\{k \, | \, G \text{ is $k$-list-colorable}\}$. 
If we require that every list contain the same set of colors, this is the usual definition of coloring. Hence, $\chi(G)\leq \chi_\ell(G)$.
If every subgraph $H$ of $G$ admits a vertex of degree at most $k$, then $G$ is said to be \emph{$k$-degenerate}. Every $k$-degenerate graph satisfies $\chi_\ell(G)\leq k+1$, since the vertices can be greedily colored. 

In a proper coloring of $G$, all vertices of a clique must have a different color as they are all adjacent to each other, which implies $\chi(G) \geq \omega(G)$.
Moreover, every color class forms an independent set, which implies $\chi(G) \geq \frac{|V(G)|}{\alpha(G)}$ for all graphs $G$.
These independent sets are cliques in the complement graph $\overline{G}$. Hence, the chromatic number of $G$ is the same as the \emph{clique covering number} of $\overline{G}$. The clique cover number of a graph $H$ is denoted by $\overline{\chi}(H)$, and so $\chi(G)=\overline{\chi}(\overline{G})$. 


The \emph{line graph} of a graph $G$, denoted $L(G)$, is the graph where the vertices correspond to the edges of $G$ and two vertices of $L(G)$ are adjacent if and only if they correspond to edges sharing a vertex in $G$. 
The class of line graphs can be characterized as the $\mathcal{F}$-free graphs where $\mathcal{F}$ is a family of 9 graphs among which is the claw, $K_{1,3}$.
A graph is a \emph{quasi-line graph} if the closed neighborhood of any vertex is the union of two cliques. (Note that there can be edges between the two cliques.) Therefore, the class of quasi-line graphs clearly does not admit a claw as an induced subgraph and this class is a proper subset of the class of claw-free graphs. Moreover, this class is a proper superset of the class of line graphs.

\paragraph{Proof strategy.}
Recall that, throughout the proof, we may restrict attention to connected graphs.
To construct colorings of claw-free graphs, we use as a road map a version of the structure theorem for claw-free graphs as presented by King and Reed~\cite[Thm~6.3]{king-reed-claw-free}, which was derived from the prior work of Chudnovsky and Seymour~\cite{claws-v-globalstructure}.
We defer the more specialized definitions of the structures and operations mentioned here until the places where we treat them.
\begin{thm}[\cite{king-reed-claw-free} Theorem 6.3]
\label{thm:structure-king-reed}
    If $G$ is a claw-free graph that has no homogeneous pair of cliques\footnote{ This theorem is a slight weakening of the statement in King and Reed~\cite{king-reed-claw-free}, which states the hypothesis for \say{skeletal} claw-free graphs instead of claw-free graphs with no homogeneous pair of cliques. However, every graph not containing a homogeneous pair of cliques is by definition skeletal. 
    }, then one of the following holds.
    \begin{enumerate}
        \item $G$ admits a clique cutset. 
        \item $G$ is a quasi-line graph. 
        \item $G$ is an antiprismatic thickening. 
        \item $G$ is a three-cliqued graph.
        \item $\chi(\overline{G}) \geq 4$ and $G$ admits
        \begin{itemize}
            \item a canonical interval $2$-join,
            \item an antihat $2$-join, 
             \item a strange $2$-join,
              \item a pseudo-line $2$-join, or
            \item a gear $2$-join.
        \end{itemize}
        \item $\chi(\overline{G}) \geq 4$  and $G$ is an icosahedral thickening.
    \end{enumerate}
\end{thm}

Our objective in this paper is to rule out the following property for any claw-free graph $G$.

\begin{property} 
    $\chi(G)> \mcodeg (G) + 3$. 
\end{property}

We say that a graph $G$ is {\em vertex-critical for \MPdot} if  $G$ satisfies \MPdot, but any proper induced subgraph of $G$ does not satisfy \MPdot.
Let $\mathcal{G}$ be a hereditary graph class, a graph class that is closed under taking induced subgraphs. Given $G \in \mathcal{G}$, we say that $G$ is {\em $\mathcal{G}$-critical for \MPdot} if  $G$ satisfies \MPdot, but any proper subgraph of $G$ in $\mathcal{G}$ does not satisfy \MPdot.  In particular, if $G$ is $\mathcal{G}$-critical for \MP for some hereditary graph class $\mathcal{G}$, then $G$ is also vertex-critical for \MPdot. 
We denote the graph class of claw-free graphs by $\mathcal{G}_{cf}$ and the class of quasi-line graphs by $\mathcal{G}_{ql}$. The classes $\mathcal{G}_{cf}$ and $\mathcal{G}_{ql}$ are both hereditary.
Our strategy is to show that no graph type listed in Theorem~\ref{thm:structure-king-reed} contains a $\mathcal{G}_{cf}$-critical graph for \MPdot.
(In the case of $\mathcal{G}_{ql}$, we use a structure theorem~\cite[Thm.~16]{ChudKingPlumSeym-rephrase-quasiline} specific to quasi-line graphs to rule out a $\mathcal{G}_{ql}$-critical graph for \MPdot.)
This implies that no claw-free graph has \MPdot, which implies that Theorem~\ref{thm:maincf} holds. Moreover, we will show that there are no $\mathcal{G}_{ql}$-critical graphs for \MPdot.

\section{Defuzzification}\label{sec:defuzz}
Quasi-line graphs and claw-free graphs can be very ``fuzzy''
---that is to say, some subgraphs, which are easy to describe, can be removed while keeping intact the core structure of the graph. The aim of this section is to define some types of fuzziness and to show that any critical graph for \MP is not fuzzy. We will warm up by dealing with two easy types of fuzziness, namely clique cutsets and dominating vertices, before handling the more complex ones: homogeneous pairs of cliques and thickenings that affect $2$-joins.
In Sections~\ref{sec:quasiline} and~\ref{sec:cf}, we prove our main results for non-fuzzy quasi-line and claw-free graphs respectively.  

\vspace{1em}
A clique cutset in a connected graph $G$ is a clique $K$ such that $G\setminus K$ is not connected. This type of fuzziness can be handled by a standard argument. We give a proof for completeness.

\begin{lem}\label{lem:no-clique-cutset} If $G$ is vertex-critical for \MPdot, then $G$ contains no clique cutset.
\end{lem}
\begin{proof}
    Suppose $G$ is vertex-critical for \MP and let $S$ be a clique cutset in $G$. 
    Let $V_1, V_2, \dots, V_t$ be the vertex sets of the components of $G \setminus S$ and let $G_i:=G[V_i\cup S]$ for each $i \in \{1,2, \dots, t\}$.
    As $G_i$ is a proper induced subgraph of the vertex-critical graph $G$,  $\chi(G_i) \leq \mcodeg(G_i) + 3 \leq \mcodeg(G) + 3$ for all $i$. 
     Since $S$ is a clique cutset, $\chi(G) = \max_i \chi(G_i) \leq \mcodeg(G)+3$, a contradiction. 
\end{proof}

Another simple but helpful result is that a vertex-critical graph for \MP does not have a vertex $x$ dominating another vertex $y$, which happens precisely when $N(y) \subseteq N[x]$.

\begin{lem}\label{lem:nodominatingvertices}  If $G$ is vertex-critical for \MPdot, then no vertex dominates another vertex.  
\end{lem}
\begin{proof}
    Suppose $G$ is vertex-critical for \MP and there exist distinct $x,y \in V(G)$ with $N(y) \subseteq N[x]$.
    Then $\deg(y) \leq \codeg(x,y) +1 \leq \mcodeg(G) + 1$. In particular, any coloring of $G-y$ with at least $\mcodeg(G)+2$ colors can be extended to $G$, by using the least available color for $y$. By vertex-criticality of $G$, it follows that $\chi(G)\le \mcodeg(G-y)+2\leq \mcodeg(G)+3$, a contradiction.
\end{proof}

\subsection{Homogeneous pairs of cliques}
A \emph{homogeneous pair of cliques} in $G$ is defined in~\cite{ChudSeym05} as a pair $(A, B)$ such that
\begin{itemize}
    \item $A, B$ are disjoint cliques in $G$,
    \item every $v \in G\setminus (A \cup B)$ is either complete or anti-complete to $A$, and $v$ is also complete or anti-complete to $B$, and
    \item $|A|+|B|\geq 3$. 
\end{itemize}
In this subsection, we show how we may assume that $G$ has no homogeneous pair of cliques.
A homogeneous pair of cliques $(A,B)$ is \emph{non-trivial}\footnote{\emph{Non-trivial} homogeneous pairs of cliques are also referred to as \emph{nonlinear}.} if there exists an induced $C_4$ in the subgraph induced by $A \cup B$.
Since $A,B$ are cliques, this is equivalent to there being an induced $C_4$ containing two vertices of $A$ and two vertices of $B$.

\begin{lem}\label{lem:trivial-homog-pair-of-cliques}
 If $G$ is vertex-critical for \MPdot, then $G$ has no trivial pair of homogeneous cliques.
\end{lem}
\begin{proof}
    Suppose $G$ is vertex-critical for \MP and admits a trivial homogeneous pair of cliques $(A,B)$. Let $J$ be the graph induced by $A \cup B$.
    Since $J$ contains no induced $C_4$, we observe that for all $x_1, x_2 \in A$ and all $x_1, x_2 \in B$, either $N(x_1) \subseteq N[x_2]$ or $N(x_2) \subseteq N[x_1]$.
    Since no vertex outside of $A \cup B$ may be mixed on $A$ or on $B$ and since $|A \cup B| \geq 3$, it follows that $G$ contains two distinct vertices $x_1, x_2$ so that $x_1$ dominates $x_2$.
    This contradicts Lemma~\ref{lem:nodominatingvertices}. 
\end{proof}
%

To see that non-trivial homogeneous pairs of cliques can be disregarded, we point out the results of Chudnovsky and Fradkin~\cite[Lem.~5.2]{ChudFrad08} and of King and Reed~\cite[Lem.~2.3]{king-reed-claw-free}. We use a slightly stronger version of the result Chudnovsky and Fradkin stated in their paper, which already holds with their proof. In this version, the graph $H$ is guaranteed to be a subgraph of $G$ rather than a minor.

\begin{lem}[\cite{ChudFrad08} Lemma 5.2 (reformulated)]\label{lem:nontrivialpairOG}
  If $G$ is a quasi-line graph that has a non-trivial homogeneous pair of cliques, then there exists
a proper subgraph $H$ of $G$ that is quasi-line and satisfies $\chi(H)=\chi(G)$. 
\end{lem}


\begin{lem}[\cite{king-reed-claw-free} Lemma 2.3 (reformulated)
]\label{lem: No non triv hom pair claw-free}
If $G$ is a claw-free graph that has a non-trivial homogeneous pair of cliques, then there exists
a proper subgraph $H$ of $G$ that is claw-free and satisfies $\chi(H)=\chi(G)$.
\end{lem}

From Lemmas~\ref{lem:trivial-homog-pair-of-cliques},~\ref{lem:nontrivialpairOG} and~\ref{lem: No non triv hom pair claw-free}, we can conclude the following.
\begin{lem}\label{lem:no-homog-pair-of-cliques}
If $G$ is a $\mathcal{G}_{cf}$- or $\mathcal{G}_{ql}$-critical graph for \MPdot, then $G$ has no pair of homogeneous cliques.
\end{lem}


\subsection{Thickenings}

 Two distinct vertices $v$ and $w$ are said to be \emph{twins} if $N[v]=N[w]$, and so $v$ dominates $w$ (and $w$ dominates $v$).
In this subsection, we give a definition of thickenings (as given in~\cite{king-reed-claw-free}) and we show that, if a graph is a strict thickening of some other graph, then it has either a homogeneous pair of cliques or a pair of twins. 

Let $G_0$ be a claw-free graph and let $M$ be a matching in $G_0$ such that $G_0-e$ is claw-free for any  $e\in M$. 
The graph $G$ is said to be a \emph{thickening of $G_0$ (under $M$)} if we can construct it from $G_0$ as follows:
\begin{enumerate}
    \item Substitute each vertex $v$ of $G_0$ by some non-empty clique $I(v)$.
    \item For every $uv \in E(G_0)$, add all edges between $I(u)$ and $I(v)$.
    \item For every $uv \in M$, remove a non-empty proper subset of the edges between $I(u)$ and $I(v)$.
\end{enumerate} 
A thickening is \emph{strict} if $|G_0| < |G|$.
For a set $S \subseteq V(G_0)$, we use $I(S)$ to denote $\bigcup_{v\in S} I(v)$. 

We now prove that any strict thickening either has a pair of twins or a homogeneous pair of cliques in specific places. 
This will then be used to show that strict thickenings do not appear in graphs which are either $\mathcal{G}_{cf}$- or $\mathcal{G}_{ql}$-critical for \MPdot.

\begin{lem}\label{lem:thickening-means-homog-or-twins}
   If $G$ is a strict thickening of some graph $G_0$, then either $G$ contains a pair of twins $a,b \in I(v)$ for some $v \in V(G_0)$ or there exists some edge $vw \in M$ such that $(I(v),I(w))$ is a homogeneous pair of cliques. 
\end{lem}
\begin{proof}
If $M = \emptyset$, then $G$ arises from $G_0$ by vertex multiplication. So for every $v \in V(G_0)$, any two distinct vertices $a,b \in I(v)$ form a pair of adjacent twins in $G$. The claim then follows from the fact that the thickening is strict.

Let $ab \in M$. By definition, $I(a), I(b)$ are disjoint cliques of $G$.
Since $M$ is a matching, no vertex of $G \setminus (I(a) \cup I(b))$ is mixed on either $I(a)$ or $I(b)$.
In order for there to be some non-empty proper subset of the edges between $I(a)$ and $I(b)$, at least one of $I(a)$ or $I(b)$ must be a clique of size at least two. Hence, $(I(a), I(b))$ is a homogeneous pair of cliques in $G$. 
\end{proof}

Since a strict thickening contains either a pair of twins (hence a vertex dominating another) or a homogeneous pair of cliques, the following is a consequence of Lemmas~\ref{lem:nodominatingvertices},~\ref{lem:no-homog-pair-of-cliques} and~\ref{lem:thickening-means-homog-or-twins}. 

\begin{cor}\label{lem:thickenings}
  If $G$ is $\mathcal{G}_{cf}$-critical or $\mathcal{G}_{ql}$-critical for \MPdot, then $G$ is not a strict thickening of some graph $G_0$.
\end{cor}

\subsection{Generalized 2-joins}
We showed that if $G$ is a $\mathcal{G}_{ql}$- or $\mathcal{G}_{cf}$-critical graph for \MPdot, then it is not a strict thickening as a whole. 
However, Theorem~\ref{thm:structure-king-reed} already hints at the fact that certain types of claw-free graphs (and of quasi-line graphs) admit so-called $2$-joins. Such operations separate the graph into two parts. We now focus on the definition of the generalized $2$-join (as given in~\cite{king-reed-claw-free}) to show that no side of a $2$-join can be a strict thickening.

    Suppose vertex sets $V_1$ and $V_2$ partition $V(G)$. For $i\in\{1,2\}$, let $G_i=G[V_i]$. Suppose there are cliques $X_i$ and $Y_i$ in $V_i$ such that $X_1 \cup X_2$ and $Y_1 \cup Y_2$ are cliques, and there are no other edges between $V_1$ and $V_2$. Then we call $((X_1 , Y_1 ), (X_2 , Y_2 ))$ a {\em generalized $2$-join} of $G$.

A reader already familiar with the terminology \emph{composition of strips}~\cite{ChudSeym05} may observe that $G$ admits a generalized $2$-join $((X_1,Y_1),(X_2,Y_2))$ if and only if $G$ is the composition of strips $(G_1, a_1, b_1)$ with $(G_2, a_2, b_2)$, where the $a_i$ and $b_i$ are simplicial vertices.

\begin{lem}\label{lem:thickening-and-$2$-join}
Suppose $G$ admits a generalized $2$-join $((X_1,Y_1),(X_2,Y_2))$ such that $G_2$ is a strict thickening of some graph $H_2$ and such that $X_2 = I(A)$ and $Y_2 =I(B)$ where $A,B$ are cliques of $H_2$. Then $G$ is neither $\mathcal{G}_{cf}$-critical nor $\mathcal{G}_{ql}$-critical for \MPdot.
\end{lem}
\begin{proof}
By Lemma~\ref{lem:thickening-means-homog-or-twins}, $G_2$ contains either a pair of twins $a,b \in I(v)$ or a homogeneous pair of cliques $I(v)$ and $I(w)$. We show that the nature of those pairs remains unchanged by the generalized $2$-join, from which we conclude that $G$ is not critical for \MPdot.

Suppose $G_2$ contains a pair of twins $a,b \in I(v)$. By the definition of $X_2$ and $Y_2$, each of these sets contains either both $a$ and $b$ or neither of them. Since the only adjacencies between vertices of $G_1$ and $G_2$ lie in the cliques $X_1 \cup X_2$ and $Y_1 \cup Y_2$, we see that $a,b$ is a pair of twins in $G$.

Suppose $G_2$ contains a homogeneous pair of cliques $I(v)$ and $I(w)$. By the definition of $X_2$ and $Y_2$, each of these sets  contains either all vertices of $I(v)$, respectively $I(w)$, or none of them. Since the only adjacencies between vertices of $G_1$ and $G_2$ lie in the cliques $X_1 \cup X_2$ and $Y_1 \cup Y_2$, every vertex of $G_1$ is either complete or anticomplete to $I(v)$ and also to $I(w)$. Hence, $I(v)$ and $I(w)$ are also a homogeneous pair of cliques in $G$.
\end{proof}

Notice that the roles of $G_1$ and $G_2$ may be exchanged in the previous result. For the different types of $2$-joins we consider later, the cliques in $G_2$ will always be of the form $X_2 = I(A)$ and $Y_2=I(B)$. Hence, Lemma~\ref{lem:thickening-and-$2$-join} will allow us to assume that $G_2$ is not a strict thickening. 



\section{Quasi-line graphs}\label{sec:quasiline}

The aim of this section is to complete the proof that the bound in Vizing's theorem (given in terms of maximum codegree) can be extended to quasi-line graphs.

\mainquasiline*

To prove this bound, we will show that no quasi-line graph is $\mathcal{G}_{ql}$-critical for \MPdot. The structure of quasi-line graphs is fully described in~\cite{ChudSeym-quasiline-cfVII} using trigraphs. However, Lemma~\ref{lem:no-homog-pair-of-cliques} ensures that no $\mathcal{G}_{ql}$-critical graph for \MP contains a homogeneous pair of cliques, and so we can use a more practical structure theorem~\cite[Thm.~16]{ChudKingPlumSeym-rephrase-quasiline}.

\begin{thm}[\cite{ChudKingPlumSeym-rephrase-quasiline} Theorem 16]\label{thm: structure quasiline}
    If $G$ is a quasi-line graph that has no non-trivial homogeneous pair of cliques, then one of the following holds. 
\begin{itemize}
    \item $G$ contains a clique cutset. 
    \item $G$ is a line graph of a multigraph.
    \item $G$ is a circular interval graph.
    \item $G$ admits a canonical interval $2$-join.
\end{itemize}
\end{thm}

By Lemma~\ref{lem:no-clique-cutset}, a graph containing a clique cutset is not vertex-critical for \MPdot, so in particular not $\mathcal{G}_{ql}$-critical. Hence, the rest of this section is dedicated to showing the result for the graphs in the other three categories. First, we show that no line graph of any (multi)graph has \MPdot, by a direct application of Vizing's theorem, which we include for completeness. Then we show that any circular interval graph satisfies $\chi_{\ell}(G) \leq \mcodeg(G)+3$, where $\chi_{\ell}$ is the list chromatic number, which implies it does not have \MPdot. Finally, we show that any graph admitting a canonical interval $2$-join is neither $\mathcal{G}_{cf}$-critical nor $\mathcal{G}_{ql}$-critical for \MPdot. Hence, no quasi-line graph $G$ is $\mathcal{G}_{ql}$-critical for \MPdot. This proves Theorem~\ref{thm:mainquasiline}. 
 
\subsection{Line graphs of multigraphs}
\begin{lem}\label{lem:linegraph_multigraphs}
    If $G$ is the line graph of some multigraph $H$, then $\chi (G) \leq \mcodeg(G) + 3$.
\end{lem}
\begin{proof}
    First, observe that if $H$ is a simple graph, then, by Vizing's theorem~\cite{vizing1964estimate}, it can be properly edge-colored with $\Delta (H) +1$ colors. For a given vertex $v\in V(H)$, all edges incident to $v$  are adjacent to each other. In particular, a pair of edges incident to $v$ is incident to $\deg(v)-2$ other edges. Hence, the corresponding vertices in $G$ have at least $\deg (v)-2$ common neighbors. Therefore, $\mcodeg(G) \geq \Delta(H)-2$ implying $\chi(G) \leq \Delta(H)+1 \leq \mcodeg(G)+3$. 
    
    If $H$ is not a simple graph, then let $H'$ be the underlying simple graph of $H$ and let $G'=L(H')$, which is $(\mcodeg(G')+3)$-colorable. 
    Note that $G'$ is an induced subgraph of $G$, where $V(G)\setminus V(G')$ is the set of vertices representing the loops and parallel edges of $H$. 
    Therefore, $\mcodeg(G')\leq \mcodeg (G)$ and there exists $\varphi$, a $(\mcodeg (G) +3)$-coloring of $G'$. 

    We are going to prove that $|N(v)|\leq \mcodeg(G)+1$ for any $v\in V(G)\setminus V(G')$, which implies that we can extend $\varphi$ to be a $(\mcodeg (G) +3)$-coloring of the whole graph $G$. 
    First, observe that the edges sharing an endpoint in $H$ form a clique in $G$. Therefore, if $v$ represents a loop at $x\in V(H)$, its neighborhood is the clique consisting of the edges that contain $x$. Hence, $|N(v)|\leq \omega(G)-1\leq \mcodeg(G) +1$. Now, if $u,v$ represent parallel edges, then $u$ and $v$ are adjacent twins in $G$. Thus, $|N(v)|=\codeg(u,v)+1\leq \mcodeg (G)+1$ also in this case.
\end{proof}

Note that $\chi_\ell(G)$ satisfies the same result as in Lemma~\ref{lem:linegraph_multigraphs} by the same argument, conditioned on the truth of the (Weak) List Coloring Conjecture.

\subsection{Circular interval graphs}\label{sub:circular}

Let $L$ be a line, and let $\mathcal{I}$ be a set of intervals of $L$.
A \emph{linear interval graph} is a graph $G$ where $V(G)$ is a finite set of distinct points on $L$ and $u,v\in V(G)$ are adjacent if they are both contained in some interval of $\mathcal{I}$. 
A \emph{circular interval graph} is defined similarly, but we replace $L$ with a circle $C$ and we let $\mathcal{I}$ be a set of arcs of $C$. 
(Note that circular interval graphs are also sometimes called circular-arc graphs in the literature.)
We call $C$ the underlying circle of $G$ and, when the situation is unambiguous, simply the \emph{cycle of $G$}.
For convenience, when the situation is unambiguous, we will conflate intervals $I \subseteq C$ with the sets of vertices of $G$ that are contained in $I$.
Throughout this paper, for every $I \in \mathcal I$, we assume without loss of generality that $I' \in \mathcal I$ for every sub-interval $I' \subseteq I$,
and that for each vertex $v$ in $G$ there is an interval in $\mathcal{I}$ containing only $v$.
The \emph{size} of an interval is the cardinality of the interval (as a set of vertices in $G$).
For any underlying cycle $C$, we will implicitly assume we have fixed an orientation of $C$ so that the directions \say{left} and \say{right} from a point on $C$ are well-defined.
Let $v \in V(G)$ and let the vertices of $G$ be ordered as $v = v_0, v_1, v_2, v_3, \dots, v_n, v_0$ when moving from $v_0$ rightwards along $C$ (where $n$ denotes the number of vertices in $G$).
The \emph{rightmost interval of $v=v_0$} is the largest interval in $\mathcal I$ of the form $\{v, v_1, \dots, v_i\}$ and we denote it by $I^R_{v}$.
Similarly, the \emph{leftmost interval of $v=v_0$} is the largest interval in $\mathcal I$ of the form $\{v_{n-i}, \ldots, v_{n-1},v_n,v\}$ and we denote it by $I^L_{v}$. 

By definition, the neighbors of a vertex $v$ in $G$ lie in $I^R_v \cup I^L_v$.
As every set of vertices lying in an interval of $\mathcal{I}$ forms a clique, every interval contains at most $\mcodeg(G) +2$ vertices of $G$. In particular, $|I^L_v|, |I^R_v| \leq \mcodeg(G)+2$. In other words, the left and right neighborhoods of $v$ both form a clique of at most $\mcodeg(G)+1$ vertices.

We prove that the negation of \MP holds in a stronger form, namely by proving that circular interval graphs are $\Delta_2(G)+2$-degenerate, proving that they are $\Delta_2(G)+3$-list-colorable. 

\begin{lem}\label{lem: degenerate circular interval}
    If $G$ is a circular interval graph, then $G$ is $\mcodeg(G)+2$-degenerate. In particular, $\chi_\ell(G)\leq \mcodeg(G)+3$. 
\end{lem}
\begin{proof}
    Let $H$ be a subgraph of $G$. Let $v$ be the rightmost vertex of a longest interval $I$, so in particular $I_v^L=I$ and $v$ is the rightmost vertex in $I_v^L$. 
    If $|H|=1$ or $\Delta(H)=0$, then the statement is trivial so suppose that $|H|\geq 2$ and $|I|\geq 2$. Let $w\in I^L_v$ be the vertex on the left of $v$. 
    We show that $|N(w)|\leq \mcodeg(G)+2$. 
    
    Let $U_w:=I^R_w-I^L_w-\{v\}$. The open neighborhood of $w$ satisfies
    $$|N(w)|=|(I^L_w\cup I^R_w)-\{w\}|=|I^L_w|+ |U_w| +|\{v\}|-|\{w\}|=|I^L_w|+|U_w|.$$ 
    Since $I_v^L \subseteq \{v\} \cup I_w^L$, the sets $U_w$ and $I_v^L$ are disjoint. Combining this with the maximality of $I^L_v$, it gives
    $$ |N(w)|=|I^L_w|+|U_w|\leq |I^L_v|+|U_w|=|I^L_v\cup U_w|.$$
    
    Note that $U_w \subseteq I^R_w$. Moreover, if $U_w \neq \emptyset$, then also $v,w \in I^R_w$. Since also $v,w \in I^L_v$, we obtain that $I^L_v\cup U_w \subseteq (N(v) \cap N(w)) \cup \{v,w\}$. Thus
    $$ |N(w)|\leq |I^L_v\cup U_w|\leq |(N(v) \cap N(w)) \cup \{v,w\}|=\codeg(v,w)+2\leq \mcodeg(G)+2.$$
   So $w$ is a vertex of $H$ of degree at most $\mcodeg(G)+2$ in $H$, concluding the proof. 
\end{proof}
 
The proof above shows that the list chromatic number of a circular interval graph $G$ satisfies $\chi_{\ell}(G) \leq \mcodeg(G)+3$. This bound is not tight. In fact, we can prove that all circular interval graphs $G$ satisfy $\chi_\ell(G) \leq \mcodeg(G)+2$, a bound that is met with equality for cliques. As this result is not necessary for our main theorems and the proof is more technical, 
we include it in Section~\ref{Section: Appendix better bound}.

\subsection{Canonical interval 2-joins}\label{ssec:canonical_interval_2join}

    Let $G$ be a linear interval graph, and let cliques $A$ and $B$ be the $|A|$ leftmost and $|B|$ rightmost vertices of $G$ in some linear interval representation of $G$. Then we call $(G, A, B)$ a \emph{linear interval strip}.
    If $(G_2 , A , B )$ is a linear interval strip where $A$ and $B$ are disjoint and such that $G_2$ is not a clique, then the $2$-join $((X_1 , Y_1 ), (A , B))$ between $G_2$ and some other graph $G_1$ with cliques $X_1$ and $Y_1$ is called a \emph{canonical interval $2$-join}~\cite{king-reed-claw-free}.


\begin{lem}\label{lem:canonical interval 2 join}
    If $G$ admits a canonical interval $2$-join $((X_1,Y_1),(X_2,Y_2))$ between $G_1$ and $G_2$, then $G$ is not vertex-critical for \MPdot. 
\end{lem}

\begin{proof}
    Suppose that $G$ is vertex-critical for \MPdot. Since $G_1, G_2$ are induced subgraphs of $G$, there exists a $(\mcodeg(G_1)+3)$-coloring $\phi$ of $G_1$. Since $\mcodeg(G_1)\leq \mcodeg(G)$, $\varphi$ is a $(\mcodeg(G)+3)$-coloring of $G_1$. We will now extend $\varphi$ to a $(\mcodeg(G)+3)$-coloring of $G$. 

    Let $L$ be the underlying line and $\mathcal{I}$ the set of intervals of $L$ in some linear interval representation of $G_2$. Then $X_2$ corresponds to the leftmost part of the leftmost interval and $Y_2$ to the rightmost part of the rightmost interval. Let $a$ be the rightmost vertex of $X_2$ on $L$ and symmetrically, let $b$ be the leftmost vertex of $Y_2$ on $L$.
    
   \begin{figure}[ht!]
        \centering
        \input{figures/CanonicalStripProof}
        \caption{Canonical interval $2$-join between $G_1$ and $G_2$.}
    \end{figure}
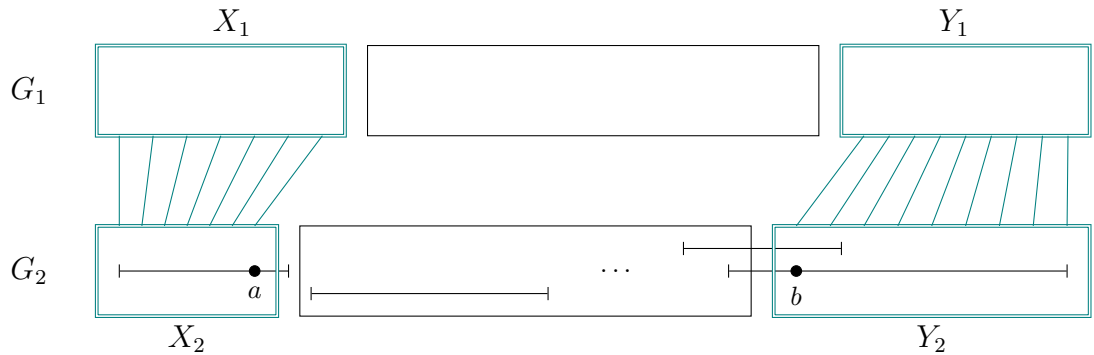
    
    We will greedily color $G_2$ by assigning the smallest available color to the vertices of $G_2$. First, we color the vertex $b$, then the other vertices of $G_2$ as they appear in $L$ from left to right. We prove for all $v \in V(G_2)$ that at most $\mcodeg(G) +2$ neighbors of $v$ were colored before $v$ itself, so at least one color remains available for $v$.
    
    First color $b$. Its colored neighborhood is the set $Y_1$. Since $b \cup Y_1$ is a clique, at most $\omega(G)-1 \leq \mcodeg(G)+1$ neighbors were colored before $b$.
    Since $X_2$ is the leftmost interval and $X_2$ and $Y_2$ are disjoint, every colored neighbor of a vertex $v \in X_2$  either lies in the clique $X_1\cup X_2$ or is the vertex $b$. Hence at most $\omega(G)-1+1\leq \mcodeg(G)+2$ neighbors of $v$ were colored before $v$ itself. 
    For any vertex $v \in G_2\setminus(X_2\cup Y_2)$, $N(v)\subseteq G_2$. So its colored neighborhood is a subset of $I^L_v\cup \{b\}$. Thus, at most $\omega(G_2)-1+1 \leq \mcodeg(G_2)+2\leq \mcodeg (G)+2$ neighbors of $v$ were colored before $v$. 
    
    Finally, for any vertex $u \in Y_2$, we have $N[u] \subseteq N[b]$. Indeed, since  $b$ is the leftmost vertex in $Y_2$, we have $N_{G_2}[u] \subseteq N_{G_2}[b]$ and the only neighbors of $u$ in $G_1$ belong to the clique $Y_1$. Thus, $|N(u)|\leq \codeg(u,b)+1\leq \mcodeg(G)+1$. 
    Therefore, we can extend $\phi$ to a $(\mcodeg(G)+3)$-coloring of $G$, and $G$ is not vertex-critical for \MPdot.
\end{proof}

\section{Claw-free graphs}\label{sec:cf}
We continue to use Theorem~\ref{thm:structure-king-reed} as a road map in our proof of Theorem~\ref{thm:maincf}.
We need to show that no claw-free graph is $\mathcal{G}_{cf}$-critical for \MPdot.
Just as for quasi-line graphs, we may assume from the results in Section~\ref{sec:defuzz} that any $\mathcal{G}_{cf}$-critical graph $G$ for \MP has no clique cutset and is not a strict thickening.
Moreover, by Theorem~\ref{thm:mainquasiline}, we may assume that $G$ is not a quasi-line graph. It remains to show that graphs in a few specific families are not $\mathcal{G}_{cf}$-critical for \MPdot, which we do in the remainder of the section. 

First, we address the class of claw-free graphs with independence number at most $2$. This family serves as a base case for antiprismatic graphs that we handle later. 
Afterwards, we deal with the three-cliqued graphs. For these graphs we will use similar techniques as for the antiprismatic graphs, however to handle the base case for these graphs we need to overcome some technicalities which lengthen the proof. Afterwards, we handle all types of  $2$-joins, and we finish up with the very small class of icosahedral graphs. Once we have shown that none of these classes contain a $\mathcal{G}_{cf}$-critical graph for \MPdot, we can conclude that Theorem~\ref{thm:maincf} holds.

\subsection{Graphs with no triad}
\label{ssec:cf_alpha2}

Let $G$ be a graph such that $\alpha(G)\leq 2$. That is, $G$ is a graph with no triad, and since a claw contains a triad, $G$ is claw-free. 
In this subsection, we prove that such $G$ satisfies $\chi(G) \leq \mcodeg(G)+2$. 
Note that $\alpha(G)\leq 2$ is equivalent to the complement $\overline{G}$ being triangle-free. 
More precisely, we will show that $\overline{\chi}(\overline{G}) \leq \macodeg(\overline{G})+2$.
For readability, let us for the remainder of this subsection write $H$ for $\overline{G}$. So we will prove that any triangle-free graph $H$ satisfies
\begin{equation}\label{eq:alpha2generalcompl}
    \overline{\chi}(H) \leq \macodeg(H)+2.
\end{equation}

Any clique covering of a triangle-free graph $H$ consists only of cliques of size $1$ and $2$. Therefore, any clique covering of $H$ is a (possibly empty) matching $M$ together with the set $X$ of vertices unmatched by $M$. 
Hence,
\begin{equation*}
    \overline{\chi}(H) = m(H)+( |V(H)|-2m(H)) =    
    |V(H)|-m(H).
\end{equation*} 
\noindent
Substituting this into Inequality~\ref{eq:alpha2generalcompl}, we see that it suffices to show that any triangle-free graph $H$ satisfies
\begin{equation}\label{eq:alpha2matching}
    |V(H)| \leq  m(H) +\macodeg(H)+2.
\end{equation}

We start by proving a lower bound on the maximum anti-codegree of a triangle-free graph in terms of the number of vertices.

\begin{lem}\label{lem: bounding anti codeg} If $H$ is a triangle-free graph, then $\macodeg(H) \geq \frac{|V(H)|}{2}- 2$.
\end{lem}
\begin{proof} Let $n=|V(H)|$. The statement is trivial if $n \leq 4$.
    We proceed by contradiction. Suppose that every pair of vertices $v,w \in V(H)$ satisfies $\acodeg(v,w) \leq \frac{n}{2}-\frac{5}{2}$. Then 
     \[|N(v) \cup N(w)| \geq\left\{\begin{array}{ll}
    (n-2)-(\frac{n}{2}-\frac{5}{2})=\frac{n}{2}+\frac{1}{2} & \text{if $vw\notin E(H)$} \\
    n-(\frac{n}{2}-\frac{5}{2})=\frac{n}{2}+\frac{5}{2}& \text{if $vw\in E(H)$}\\
    \end{array}.\right.\] 

    Let $x$ and $y$ be two other vertices. Then $|N(x) \cup N(y)| \geq \lceil \frac{n}{2}+\frac{1}{2} \rceil \geq 3$, so there exist two vertices $s$ and $t$, where $s \in \{x,y\}$ and $t \in [N(x) \cup N(y)]-\{v,w\}$ such that $s$ is adjacent to $t$. Then $|N(v) \cup N(w)|+|N(s) \cup N(t)| \geq n+5$. So five vertices of $G$ appear in both of $N(v) \cup N(w)$ and $N(s) \cup N(t)$, where $s,t,v$ and $w$ can also appear (multiple times) in these sets.  Since each of these five vertices must be adjacent to $v$ or $w$ and be adjacent $s$ or $t$, by the pigeonhole principle some two vertices must have two common neighbors in $v,w,s,t$.
    Hence, $H$ contains $C_4$ as a subgraph.
       
    Let $v_1,v_2,v_3,v_4$ be a $C_4$ in $H$. Then, since $H$ is triangle-free and $v_iv_{i+1} \in E(H)$ for all $i=1,\ldots, 4$ (where $v_5=v_1$), we have that $N(v_i) \cap N(v_{i+1}) = \emptyset$. We arrive at a contradiction with
    \begin{align*} n &\geq |N(v_1)\cup N(v_2) \cup N(v_3) \cup N(v_4)|= |N(v_1) \cup N(v_3)| + |N(v_2)\cup N(v_4)|\\ &\geq \frac{n}{2}+\frac{1}{2} + \frac{n}{2}+\frac{1}{2} = n+1. \qedhere\end{align*}
\end{proof}

Observe that $H= K_{\frac{n}{2}, \frac{n}{2}}$ satisfies $\macodeg(H) = \frac{n}{2}- 2$. Thus, the bound proven in Lemma~\ref{lem: bounding anti codeg} is tight.

\begin{lem}\label{lem: no perfect matching}
    If $H$ is a triangle-free graph such that $2m(H) \leq |V(H)|-2$, then \[|V(H)| \leq m(H)+\macodeg(H)+2.\]
\end{lem}
\begin{proof}
    Let $M$ be a maximum matching in $H$ and let $X$ be the set of vertices unmatched by $M$. Then $X$ is an independent set and  $|M|=m(H)$. Moreover, as $2m(H) \leq |V(H)|-2$, $X$ contains at least two vertices, say $v$ and $w$.
    
    We claim that at least one endpoint of any edge of $M$ is a non-neighbor of both $v$ and $w$. Indeed, suppose for a contradiction that there exists $ab$, an edge of $M$ such that $a,b\in N(v)\cup N(w)$. As $H$ is triangle-free, $a$ and $b$ are adjacent to distinct vertices in $\{v,w\}$. Then there is a matching of size $2$ between $\{a,b\}$ and $\{v,w\}$ and replacing $ab$ by this matching in $M$ leads to a larger matching in $H$, a contradiction.

    Hence, $\acodeg(v,w)$ counts at least one endpoint of every edge of $M$ and every vertex of the independent set $X$, except for $v$ and $w$.
    The other endpoint of an edge in $M$ is accounted for by $m(H)$. So every vertex is counted (at least once) in $m(H) + \acodeg(v,w)+2$ as desired.
\end{proof}

\begin{thm}\label{thm:alpha2_chiBounded}
    If $G$ is a graph such that $\alpha(G)\le2$, then $\chi(G) \leq \mcodeg(G)+2$.
    Equivalently, if $H$ is triangle-free, then $\overline{\chi}(H) \leq \macodeg(H)+2$.
\end{thm}
\begin{proof}
As we argued at the beginning of the subsection, it suffices to establish Inequality~\ref{eq:alpha2matching}, where $H$ is the (triangle-free) graph $\overline{G}$.
By Lemma~\ref{lem: no perfect matching}, we may assume $2m(H) \geq |V(H)|-1$.
Then we obtain using Lemma~\ref{lem: bounding anti codeg} that
\[m(H) + \macodeg(H)+2 \geq \frac{|V(H)|}{2}-\frac{1}{2}+\frac{|V(H)|}{2}-2+2=|V(H)|-\frac{1}{2}.\]
Since both expressions $m(H) + \macodeg(H)+2$ and $|V(H)|$ are integers, we obtain that also 
$m(H)+\macodeg(H)+2 \geq |V(H)|$, as desired.
\end{proof}
  
\subsection{Antiprismatic graphs}\label{ssec:antiprismatic}
A graph is {\em antiprismatic} if it does not contain one of $\{K_{1,3}, 2K_1 + K_2, 4K_1\}$ as an induced subgraph. Such graphs are alternatively described by Chudnovsky and Seymour~\cite{ChudSey07I,ChudSey07II} through their complements, the prismatic graphs. A graph is \emph{prismatic} if for every triangle $T$ and vertex $v\notin T$, $v$ is adjacent to precisely one vertex in $T$. 
In this subsection, we give a bound on the clique covering number of the prismatic graphs in order to give a bound on the chromatic number of the antiprismatic graphs.
To do so, we reduce the problem to the case $\alpha \leq 2$.

\begin{lem}\label{lem: prismatic}
    If $H$ is a prismatic graph, then $\overline{\chi}(H) \leq \macodeg(H) + 2$.
\end{lem}
\begin{proof}
    Let $H$ be a prismatic graph and let $\mathcal{T}$ be a set of disjoint triangles of $H$ such that $H-\mathcal{T}$ is triangle-free. 
    We prove the statement by induction on $|\mathcal{T}|$.

    The base case is $|\mathcal{T}|=0$. Then $H$ itself is triangle-free and the result follows from Theorem~\ref{thm:alpha2_chiBounded}. 

    Suppose now that $H$ is a prismatic graph such that $H-\mathcal{T}$ is triangle-free with $|\mathcal{T}|\geq 1$.
    Let $T\in \mathcal{T}$. Then $H':=H-T$ is a prismatic graph, since the class of prismatic graphs is closed under taking induced subgraphs. Now $\mathcal{T'}:=\mathcal{T}-T$ is a set of triangles such that $H'-\mathcal{T'}$ is triangle-free with $|\mathcal{T'}|<|\mathcal{T}|$. Then $\overline{\chi}(H')\leq \macodeg(H') +2$ by the induction hypothesis. 
    As $T$ is a triangle, only one clique suffices to cover it. Hence, $\overline{\chi}(H)\leq \macodeg(H')+1+2$. 

    We now show that adding a triangle to a prismatic graph increases the maximum anticodegree by at least $1$. Let $u,v \in H'$ such that $\acodeg(u,v)=\macodeg(H')$. Since $H$ is a prismatic graph and $u, v \notin T$, $u$ and $v$ are both adjacent to exactly one vertex of $T$. Hence, there is at least one vertex of $T$ which is neither adjacent to $u$ nor to $v$. Thus, $\macodeg(H')+1=\acodeg^{H'}(u,v)+1 \leq \acodeg^{H}(u,v) \leq \macodeg(H)$. 
    Therefore, we conclude that $\overline{\chi}(H)\leq \macodeg(H')+1+2\leq \macodeg(H) +2$, as promised.
\end{proof}

Taking complements, we obtain the following corollary.

\begin{cor}\label{cor: antiprismatic}
    If $G$ is an antiprismatic graph, then $\chi(G) \leq \mcodeg(G)+2$.
\end{cor}

\subsection{Three-cliqued graphs}
\label{ssec:three-cliqued}

Let $G$ be a claw-free graph such that $\chi(\overline{G}) \leq 3$. Then there exists a clique cover $(K_1,K_2,K_3)$ of $G$, and $G$ is said to be \emph{three-cliqued}. 
In this section, we show that no such $G$ satisfies \MP by proving a stronger result. We establish that the chromatic number is bounded from above by the maximum codegree of a pair of vertices in the same $K_i$, $i\in\{1,2,3\}$ plus 3. First, we show that this result holds when the graph has no triad by working in its triangle-free complement. Then we extend it to all graphs by showing that any three-cliqued vertex-critical graph for \MP does not contain a triad.

We start with the following bound. This is tight for any $3$-coloring of $C_5$ and for the $3$-coloring of $C_6$ where diagonal vertices have the same color. The proof uses the same ideas as in Subsection~\ref{ssec:cf_alpha2}, but here we need to bound the codegree of two vertices satisfying $c(u)=c(v)$. This extra constraint introduces complications that make the proof significantly longer.

\begin{lem}\label{lem:alpha2_strong_triangleFree}
    Let $H$ be a triangle-free graph that admits a proper $3$-coloring $c$. Then there exist two vertices $u,v$ of $H$ with $c(u)=c(v)$ such that $\overline{\chi}(H) \leq \acodeg(u,v)+3$.
\end{lem}
\begin{proof}
    Since $H$ is triangle-free, any clique cover of $H$ consists of cliques of size $1$ and $2$. Just as in Subsection~\ref{ssec:cf_alpha2}, we have  
    $\overline{\chi}(H) = |V(H)|-m(H)$. So analogously to Inequality~\ref{eq:alpha2matching}, we need to prove 
    \begin{equation}\label{eq:triangleFree_mcodegCliqueBegin}
    |V(H)| \leq m(H) + \max_{c(u)=c(v)} \acodeg(u,v)+3.
    \end{equation}
    
    Let $M$ be a maximum matching of $H$ and let $X$ be the set of  vertices unmatched by $M$. Then $X$ is an independent set and $|V(H)|=2m(H)+|X|$. We derive from this two useful inequalities equivalent to Inequality~\ref{eq:triangleFree_mcodegCliqueBegin}. First,

    \begin{align}
        &2m(H)+|X| \leq m(H) + \max_{c(u)=c(v)} \acodeg(u,v)+3 \notag \\
        \iff & m(H) +|X|-3\leq \max_{c(u)=c(v)} \acodeg(u,v).\label{eq:maxcodeg_mX}
    \end{align}

Second, observing that $|V(H)| \equiv |X| (\mod 2)$, we have
    \begin{align}
&|V(H)| \leq \frac{|V(H)|-|X|}{2} + \max_{c(u)=c(v)} \acodeg(u,v) + 3 \notag \\
 \iff &  \frac{|V(H)|+|X|}{2}-3 \leq \max_{c(u)=c(v)} \acodeg(u,v) \notag\\
 \iff &  \frac{|V(H)|+|X|}{2}-3.5 \leq \max_{c(u)=c(v)} \acodeg(u,v), \label{eq:triangleFree_mcodegClique_odd}
 \end{align} where the last equivalence holds as the right-hand side is integral by definition.

\begin{claim}\label{cl:multicolorX}
        If there are two vertices $u,v\in X$ such that $c(u)=c(v)$, then Inequality~\ref{eq:triangleFree_mcodegCliqueBegin} holds.
    \end{claim}
    
    \begin{proofclaim}
        Suppose that there are two unmatched vertices $u,v$ such that $c(u)=c(v)$. Let $e=ab$ be any edge in $M$. We claim that at least one of the vertices is a common anti-neighbor of $u$ and $v$. Indeed, since $H$ is triangle-free, $u$ is adjacent to at most one vertex of $a$ and $b$, and so is $v$. Hence, if neither $a$ nor $b$ is an anti-neighbor of both $u$ and $v$, we have without loss of generality that $au, bv \in E(H)$. But now we can construct a larger matching than $M$ by replacing the edge $ab$ with edges $au$ and $bv$, a contradiction.

        Therefore, $\acodeg(u,v)$ counts at least one endpoint of every edge of $M$. Moreover, $\acodeg(u,v)$ accounts for all unmatched vertices except for $u,v$. The other endpoint of each edge in the matching is counted by $|M| = m(H)$. Thus, every vertex is counted by $m(H) + \acodeg(u,v)+2$, which confirms Inequality~\ref{eq:triangleFree_mcodegCliqueBegin}.
    \end{proofclaim}

Therefore, we can assume that every maximum matching $M$ leaves at most one vertex per color class unmatched. In particular, $|X| \leq 3$. So to prove Inequality~\ref{eq:triangleFree_mcodegClique_odd}, and hence Inequality~\ref{eq:triangleFree_mcodegCliqueBegin}, it is sufficient to prove the following:
\begin{align}
\frac{|V(H)|}{2}-2 \leq \max_{c(u)=c(v)} \acodeg(u,v).\label{eq:triangleFree_mcodegClique_X3} 
 \end{align}

A bicolored $C_4$ is a cycle on which the proper coloring $c$ uses only two colors, i.e.~the diagonally opposite vertices have the same color.

\begin{claim}\label{claim:bicolored_C4}
    If $H$ admits a bicolored $C_4$, then  Inequality~\ref{eq:triangleFree_mcodegCliqueBegin} holds.
\end{claim}

\begin{proofclaim}
As just discussed, it suffices to establish Inequality~\ref{eq:triangleFree_mcodegClique_X3}.
    Suppose for a contradiction that $\frac{|V(H)|}{2}-2>\acodeg(u,v)$ for all vertices $u, v$ satisfying $c(u)=c(v)$, so  $\frac{|V(H)|}{2}-2.5 \geq \acodeg(u,v)$. Since vertices in the same color class are nonadjacent, we also have, for such $u,v$, that
    \[|N(u) \cup N(v)| = |V(H)|-2-\acodeg(u,v)\geq \frac{|V(H)|}{2}+0.5.\]

    Let $u,x,v,y$ be a bicolored $C_4$. It follows from the above observation that $|N(u) \cup N(v)|+ |N(x) \cup N(y)| \geq |V(H)|+1$. Hence, there exists a vertex $z$ that belongs to both $N(u) \cup N(v)$ and $N(x) \cup N(y)$. Without loss of generality, we may assume $z \in N(u)$ and $z \in N(x)$, but now $zux$ forms a triangle, a contradiction.
\end{proofclaim}

We can now complete the proof for all $H$ such that $|V(H)| \geq 22$.

For convenience, suppose yellow, red, and blue are the three colors used by $c$ on $H$ and suppose, without loss of generality, that their corresponding color classes $Y,R,B$ in $H$ satisfy $|Y| \geq |R| \geq |B|$.

\begin{claim} Inequality~\ref{eq:triangleFree_mcodegCliqueBegin} holds for all graphs $H$ satisfying $|V(H)| \geq 22$.
\end{claim}

\begin{proofclaim}
By Claim~\ref{claim:bicolored_C4}, we may assume that $H$ contains no bicolored $C_4$.
As in the proof of Claim~\ref{claim:bicolored_C4}, suppose for a contradiction that $\frac{|V(H)|}{2}-2>\acodeg(u,v)$ for all vertices $u, v$ satisfying $c(u)=c(v)$.
Since $|V(H)| \geq 22$, we have $|Y| \geq 8$. Letting $u,v,x,y \in Y$, it follows just as in the proof of Claim~\ref{claim:bicolored_C4} that
\[|N(u) \cup N(v)|+ |N(x) \cup N(y)| \geq |V(H)|+1.\]
As $Y$ is an independent set, this inequality implies that there are at least $|Y|+1$ vertices that belong to both $N(u) \cup N(v)$ and $N(x) \cup N(y)$. Notice that there are $2 \cdot 2=4$ ways to select a vertex from each side. Since $|Y|\geq 8$, we have $|Y|+1 \geq 9$. Therefore, there are three vertices that have the same pair of neighbors in $Y$. Hence, at least two of them have the same color. These two vertices together with their neighbors in $\{t,u\} \times \{v,w\}$ form a bicolored $C_4$, a contradiction. 
\end{proofclaim}

The remainder of the proof is devoted to the case $|V(H)| \leq 21$. We need a special choice of matching $M$. Let $M$ be a maximum matching such that $M$ minimizes $BR$-edges, and among those, the one that maximizes $YR$-edges.

    If $V(H) \leq 4$ then Inequality~\ref{eq:triangleFree_mcodegClique_X3} is trivially true. For $V(H) > 4$, if $|R|\leq 1$ then $|Y| \geq V(H)-2$, implying $\max_{c(u)=c(v)} \acodeg(u,v)\geq |Y|-2 = V(H)-4 \geq \frac{|V(H)|}{2}-2$. So also in that case the inequality is trivially true. Hence, we can assume $|R|\geq 2$ from now on. Finally, if $\max_{c(u)=c(v)} \acodeg(u,v)\geq m(H)$, the inequality holds, because $|X|\leq 3$.
    
    \begin{claim}
        If $X$ contains a vertex of $Y$ and $|Y|\geq 4$, then  Inequality~\ref{eq:triangleFree_mcodegCliqueBegin} holds.
        In particular, if $|V(H)|\geq 10$ and $|X|=3$, then  Inequality~\ref{eq:triangleFree_mcodegCliqueBegin} holds.
    \end{claim}
    
    \begin{proofclaim}
        Suppose $u \in Y\cap X$ and let $v \in Y\setminus\{u\}$. 
        As $Y$ is an independent set, $u,v$ are both not adjacent to the yellow endpoints of the $YR$- and $YB$-edges of $M$. 

        The vertex $u$ is not adjacent to any endpoints of the $BR$-edges of $M$, as otherwise we could have replaced this edge with one incident to $u$, obtaining a maximum matching $M'$ with fewer $BR$-edges.
        Since $H$ is triangle-free, the vertex $v$ is adjacent to at most one endpoint of every $BR$-edge in $M$. Thus, each of $u,v$ is non-adjacent to at least one endpoint of each edge in $M$ not containing $v$. Hence $\acodeg(u,v) \geq m(H)-1$.  

        If $|X|\leq 2$, then $V(H)= 2m(H)+2$ and thus Inequality~\ref{eq:maxcodeg_mX} holds. 
        Similarly, if $\macodeg(H) \geq m(H)$, then Inequality~\ref{eq:maxcodeg_mX} also holds. 

        Therefore, suppose that $\macodeg(H)=m(H)-1$ and $|X|=3$.
        For this equality to hold, any vertex $w \neq u$ in $Y$ must be non-adjacent to exactly one endpoint of each edge in $M$ that does not contain $w$. Thus, $w$ is adjacent to exactly one endpoint of each edge in $M$ not containing $w$. So if $xy$ is an edge in $M$ where $y\in Y$, then $wx \in E(H)$.

        As $|Y\setminus\{u\}|\geq 3$, there are two edges $w_1a$ and $w_2b$ in $M$ such that $a$ and $b$ have the same color. By the preceding paragraph, we obtain that $w_2a, w_1b \in E(H)$. Thus, $w_1aw_2b$ is a bicolored $C_4$, and we are done by Claim~\ref{claim:bicolored_C4}.
    \end{proofclaim}

    \begin{claim}
        If $|Y|\geq 4$, $|R|\geq 2$ and $X$ contains a vertex of $R$, then  Inequality~\ref{eq:triangleFree_mcodegCliqueBegin} holds.
        In particular, if $|V(H)|\geq 10$ and $|X|=2$,  Inequality~\ref{eq:triangleFree_mcodegCliqueBegin} holds.
    \end{claim}
    
    \begin{proofclaim}
        By the previous claim, we can assume $Y\cap X=\emptyset$ implying $|X|\leq 2$.

        Let $v\in R\cap X$. First, we show that if there is a $BR$-edge in $M$, then  Inequality~\ref{eq:triangleFree_mcodegCliqueBegin} holds. 
        
        Let $w \in R$ lie in a $BR$-edge of $M$. Suppose that there exists some $BY$-edge such that none of its endpoints is a common non-neighbor of $v$ and $w$. Then we must have $w$ adjacent to the $Y$ and $v$ adjacent to the $B$ endpoint, as $w$ cannot be adjacent to both endpoints. Hence, we can replace the $wB$ and $BY$-edge of $M$ with the $vB$ and $wY$ edges. Then the number of edges and also $BR$-edges stays the same, but the number of $YR$-edges increases. This is impossible by our choice of $M$.

        Hence, every $BY$-edge of $M$ contains a common non-neighbor of $v$ and $w$. Since the red vertices form an independent set, we have $\acodeg(v,w) \geq m(H)-1$ and Inequality~\ref{eq:maxcodeg_mX} is verified as $|X|\leq2$.

        Now suppose that there exists no $BR$-edge in $M$. Then we see that $m(H) \leq |Y|\leq \acodeg (u,v)+2$ with $u,v\in Y$, implying
        \[m(H) + \max_{c(u)=c(v)} \acodeg(u,v)+3 \geq 2m(H)-2+3 \geq 2m(H)+1.\]
        Thus, Inequality~\ref{eq:triangleFree_mcodegCliqueBegin} is satisfied if $|X|\leq 1$ or if $m(H)<|Y|$. Therefore, suppose that $m(H)=|Y|=\frac{|V(H)|}{2}-1$, and that $N(u) \cup N(v) = B \cup R$ for all $u,v \in Y$. Similar to above, if $|V(H)| \geq 10$ then $|Y| \geq 4, |R| \geq 2$ implying that $H$ contains a bicolored $C_4$ and we are done by Claim~\ref{claim:bicolored_C4}.    
    \end{proofclaim}

    \begin{claim}
        If $|Y|\geq 6$ and $|R|\geq 2$, then  Inequality~\ref{eq:triangleFree_mcodegCliqueBegin} holds.
        In particular, if $|V(H)|\geq 16$,  Inequality~\ref{eq:triangleFree_mcodegCliqueBegin} holds.
    \end{claim}
    
    \begin{proofclaim}
        As $|Y|\geq 4$ and $|R|\geq 2$, $X$ does not contain a vertex from $Y\cup R$ as otherwise  Inequality~\ref{eq:triangleFree_mcodegCliqueBegin} holds by the two previous claims. 
        
       As in the proof of Claim~\ref{claim:bicolored_C4}, $|X|\leq 1$ implies that to prove Inequality~\ref{eq:triangleFree_mcodegClique_odd} it suffices to show the following: \begin{align}\label{main ineq}
        \frac{|V(H)|}{2}-3 \leq \max_{c(u)=c(v)} \acodeg(u,v). 
        \end{align}

        For a contradiction, suppose that $\frac{|V(H)|}{2}-3>\acodeg(u,v)$ for all vertices $u$ and $v$ satisfying $c(u)=c(v)$, so  $\frac{|V(H)|}{2}-3.5 \geq \acodeg(u,v)$. Since $u$ and $v$ are in the same color class, they are not adjacent. Hence, for all $c(u)=c(v)$,
    \[|N(u) \cup N(v)| = |V(H)|-2-\acodeg(u,v)\geq \frac{|V(H)|}{2}+1.5.\]

        If $|V(H)| \geq 16$, then $|Y| \geq 6$. Let $u,v,x,y \in Y$. Then
        \[|N(u) \cup N(v)|+ |N(x) \cup N(y)| \geq |V(H)|+3.\]
        Hence, there are at least $|Y|+3 \geq 9$ vertices in both $N(u) \cup N(v)$ and $N(x) \cup N(y)$. Notice that there are $2 \cdot 2=4$ ways to select a vertex from each side. By the pigeonhole principle, at least three vertices have the same pair of neighbors in $Y$. Hence, at least two of them have the same color. These two vertices together with their neighbors form a bicolored $C_4$. Thus, by Claim~\ref{claim:bicolored_C4},  Inequality~\ref{eq:triangleFree_mcodegCliqueBegin} holds.
    \end{proofclaim}

    To complete the proof, we carry out a case analysis. First, we cover instances where $|V(H)|\geq 10$ and $|X| \leq 1$, then when $|V(H)|\leq 9$ and $|X|\leq3$. For each remaining case, we proceed by contradiction. Suppose that Inequality~\ref{main ineq} does not hold. Notice that no vertex has a neighbor in its own color.
    
\begin{itemize}
    \item $|V(H)|=15$. Since $|Y| \leq 5$, we have $|Y|=|R|=|B|=5$. Since Inequality~\ref{main ineq} does not hold, we have $\max_{c(u)=c(v)} \acodeg(u,v) \leq \frac{|V(H)|}{2}-3 = 4.5$. So $|N(u) \cup N(v)| \geq 15-2-4 =9$ for all $c(u)=c(v)$. Hence, there exists $v \in Y$ such that $|N(v)| \geq 5$ and such that $v$ has at least three neighbors $v_1,v_2,v_3$ of the same color, say blue. Since $H$ does not contain a bicolored $C_4$ in the colors yellow and blue, $N(v_i) \cap N(v_j) \cap Y = \{v\}$ for $i\ne j$. Then by the pigeonhole principle applied to the neighborhoods of the $v_i$'s, there are two vertices $v_i$ and $v_j$, say $v_1$ and $v_2$, such that $N(v_1) \cup N(v_2)$ contains at most three yellow vertices. Then $|N(v_1) \cup N(v_2)| \leq 3+|R|= 8$, a contradiction as $v_1$ and $v_2$ are both blue.
    \item $|V(H)|=14$. In this case, we have $|Y|=|R|=5$ and $|B|=4$ and $|N(u) \cup N(v)| \geq 9$ for all $c(u)=c(v)$. Hence, for all $u,v\in Y$, we have $N(u) \cup N(v)= B \cup R$. Thus, every red vertex is non-adjacent to at most one yellow vertex. Hence, every pair of red vertices has two common yellow neighbors implying that there is a bicolored $C_4$, a contradiction to Claim~\ref{claim:bicolored_C4}.
    \item $|V(H)|=13$. In this case, we have $|Y|=5$ and $|N(u) \cup N(v)| \geq 8$ for all $c(u)=c(v)$. Thus, for all $u,v\in Y$, $N(u) \cup N(v)= B \cup R$ which leads, just as in the previous case, to a bicolored $C_4$.
    \item $|V(H)|=12$. In this case, we have $|N(u) \cup N(v)| \geq 8$ for all $c(u)=c(v)$. Hence, every color class has size $4$. Hence, again for all $u,v\in Y$ $N(u) \cup N(v)= B \cup R$ leading to a bicolored $C_4$.
    \item $|V(H)|=11$. In this case, we have $|N(u) \cup N(v)| \geq 7$ for all $c(u)=c(v)$. Hence, every color class has size at most $4$, and so $|Y|=|R|=4$ and $|B|=3$. Once again $N(u) \cup N(v)= B \cup R$  for all $u,v \in Y$ leading to a bicolored $C_4$.
    \item $|V(H)|=10$. In this case, $|Y| \geq 3$, but  $|N(u) \cup N(v)| \geq 7 > |B \cup R|$, a contradiction.
\end{itemize}
For $|V(H)|\leq 9$, we may not assume that $|X|\leq 1$, but Claim~\ref{cl:multicolorX} still holds. So in particular $|X|\leq 3$. Notice that, by the definition of $X$, we have $|X| \equiv V(H) \pmod 2$.
\begin{itemize}
    \item If $|V(H)| \leq 4$, then $\frac{|V(H)|}{2}-2 \leq 0$ so then Inequality~\ref{eq:triangleFree_mcodegClique_X3} is automatically satisfied.
    \item $|V(H)| =5$. If $|X|=1$ then Inequality~\ref{main ineq} is satisfied as $\frac{|V(H)|}{2}-3 \leq 0$. Thus, assume $|X|=3$ and $m(H)=1$. Since $H$ is also triangle-free, it is a subgraph of $K_{1,4}$. Now, two of the leaves of $K_{1,4}$ have the same color and anti-codegree at least 2. Hence,  Inequality~\ref{eq:triangleFree_mcodegClique_X3} is satisfied.
    
    \item $|V(H)|=6$. If $|X|=0$ then Inequality~\ref{main ineq} is satisfied as $\frac{|V(H)|}{2}-3 \leq 0$. Hence, $|X|=2$ and $m(H)=2$. If Inequality~\ref{eq:triangleFree_mcodegClique_X3} is not satisfied, then $\max_{c(u)=c(v)} \acodeg(u,v)=0$. Hence, every color class has at most two vertices. Moreover, if $u,v \in Y$ then $N(u) \cup N(v) = B \cup R$ and if $u,v \in R$ then $N(u) \cup N(v) = B \cup Y$. Hence, there are two vertex disjoint edges between $Y$ and $R$. So there exists a maximal matching with two $YR$-edges. Hence, this maximum matching leaves two blue vertices unmatched implying that the main inequality is satisfied by Claim~\ref{cl:multicolorX}.
    \item $|V(H)|=7$. Then $|Y| \geq 3$ implying  $\max_{c(u)=c(v)} \acodeg(u,v) \geq 1$. Hence, if $|X|=1$ then Inequality~\ref{main ineq} is satisfied. So assume $|X|=3$ so that each color class contains exactly one unmatched vertex. If the main inequality is not satisfied, we have $\max_{c(u)=c(v)} \acodeg(u,v)=1$ and $m(H)=2$.
    
    In particular, we have $|Y|=3$ and the first edge in the matching is $YR$ and the other is $YR$ or $YB$. Let $b$ be the unmatched blue vertex. Then $b$ is not adjacent to the yellow vertex of $YR$ edge in $M$ as otherwise, there is a maximum matching with two unmatched red vertices, a contradiction. $b$ is also not adjacent to the unmatched yellow vertex. Hence, the unmatched yellow vertex and the yellow endpoint of the $YR$ edge in $M$, have two common non-neighbors the third yellow vertex and $b$. Hence, Inequality~\ref{eq:triangleFree_mcodegClique_X3} is satisfied.
    
    \item $|V(H)|=8$. Then $|Y| \geq 3$, implying  $\max_{c(u)=c(v)} \acodeg(u,v)=1$. Therefore, if $|X|=0$, then Inequality~\ref{main ineq} holds.  So assume $|X|=2$. If Inequality~\ref{eq:triangleFree_mcodegClique_X3} is not satisfied, we have $\max_{c(u)=c(v)} \acodeg(u,v)=1$ and $m(H)=3$.
    
    Hence, any color class has at most three vertices. So $|Y|=|R|=3$ and $|B|=2$. In particular, $N(u) \cup N(v) = R \cup B$ for all $u,v \in Y$ and $N(s) \cup N(t) = Y \cup B$ for all $s,t \in R$. Hence, every blue vertex is adjacent to at least two vertices of both $R$ and $Y$.    
    Let $b \in B$ be adjacent to $u,v \in Y$ and $s,t \in R$. Since $H$ is triangle-free, we have that $u,v$ and $s,t$ form an independent set. But then $s,t \not \in N(u) \cup N(v)$, a contradiction.

    \item $|V(H)|=9$. Note that $m(H) \geq 3$. Therefore, if $|Y| \geq 5$, we are done. We need to prove that  $m(H)+\max_{c(u)=c(v)} \acodeg(u,v)+3\geq 9$. We do a separate case distinction based on whether $|Y|=3$ or $|Y|=4$, and whether $m(H)=3$ or $m(H)=4$, so if $|X|=3$ or $|X|=1$. Notice that if $|Y|=3$, then $|Y|=|B|=|R|=3$.
    \begin{itemize}
    \renewcommand\labelitemii{$\circ$}
        \item $|Y|=4$ and $m(H)=4$. In this case, $\max_{c(u)=c(v)} \acodeg(u,v)\geq 2$ implying $m(H)+\max_{c(u)=c(v)} \acodeg(u,v)+3 \geq 4+2+3=9$ as desired.
        \item $|Y|=4$ and $m(H)=3$. If the inequality is not satisfied, then $\max_{c(u)=c(v)} \acodeg(u,v)= 2$. So $N(u) \cup N(v) = B \cup R$ for all $u,v\in Y$. Hence, every non-yellow vertex is non-adjacent to at most one yellow vertex. Since $|R| \geq 3$, two red neighbors have at least two common neighbors. Hence, $H$ contains a bicolored $C_4$ in yellow and red.
        \item $|Y|=3$ and $m(H)=4$. If the inequality does not hold, then $\max_{c(u)=c(v)} \acodeg(u,v) = 1$. Hence, if $c(u)=c(v)$, then $N(u) \cup N(v)$ contains all vertices of the other color classes. So every vertex is non-adjacent to at most one vertex in another color class.
        Since there is no bicolored $C_4$, every vertex is adjacent to exactly two vertices in each of the other color classes. Let $v \in Y$ be adjacent to $s,t \in R$ and $a,b \in B$. Since $H$ is triangle-free, we have that $s,t,a,b$ must form an independent set. So $\acodeg(a,b) \geq \{s,t\} +\{c\} = 3$, a contradiction.
        \item $|Y|=3$ and $m(H)=3$. If the inequality does not hold, then $\max_{c(u)=c(v)} \acodeg(u,v) \leq 3$.
        So each color class contains one unmatched vertex $
        w\in Y$, $r \in R$ and $a\in B$. Moreover, the matching consists of edges $bs = BR, yc \in YB$ and $xt \in YR$. Now, we have that $w$ is non-adjacent to $r$, since $M$ is maximal.
        If $y$ is adjacent to $r$, then we can replace the edge $yc$ with $yr$ to have a maximum matching where two vertices of $B$ are unmatched, implying that the inequality holds. So $y$ is not adjacent to $r$. Thus, both $r$ and $x$ are common non-neighbors of $y$ and $w$. So every other blue or red vertex is in $N(w) \cup N(y)$. 

        If $w$ is adjacent to $b$ or $w$ is adjacent to $s$, then we can replace the edge $bs$ to obtain a maximum matching where two vertices of the same color are unmatched, a contradiction. So $y$ is adjacent to $b$ and $y$ is adjacent to $s$, implying that $ybs$ is a triangle, a contradiction. \qedhere
    \end{itemize}
\end{itemize}
\end{proof}

We now bound the chromatic number of $G$ by the maximum codegree of two vertices in the same $K_i$ of the clique covering $(K_1,K_2,K_3)$ of $G$.

\begin{lem}\label{lem: updatedstatement}
    Let $G$ be a claw-free graph and let $(K_1,K_2,K_3)$ be a clique cover of $G$. Then there exists $i\in \{1,2,3\}$ such that \[\chi(G)\leq  \max_{u,v \in K_i} \codeg(u,v)+3.\]
\end{lem}

\begin{proof}
For a contradiction, suppose there exists some claw-free $G$ with a clique cover $(K_1,K_2,K_3)$ that is a minimal counterexample, i.e.~$\chi(G)> \max_{u,v \in K_i} \codeg(u,v) +3$. 

%

By Lemma~\ref{lem:alpha2_strong_triangleFree}, we may assume that $\alpha(G) \ge 3$.
So, $G$ contains a triad, say $T$, and $\chi(G)\leq \chi(T)+\chi(G-T)=1+\chi(G-T)$. Moreover, since $(K_1,K_2,K_3)$ is a clique cover for $G$, $T$ contains one vertex of $K_i$ for each $i$.

Suppose first that $G$ does not contain a $K_i$ with at least three vertices. By the pigeonhole principle, $|V(G)|\leq 6$, and $|G-T|\leq 3$. As $G$ is a counterexample, $\max_{u,v \in K_i} \codeg(u,v)+3<\chi(G)\le \chi(G-T)+1$. From this, we have that $G-T$ is a triangle, $|K_1|=|K_2|=|K_3|=2$ and $\max_{u,v \in K_i} \codeg(u,v)=0$. 
The triad consists of $a_1,a_2,a_3$ and the other three vertices $b_1, b_2,b_3$ form a triangle such that $a_i,b_i \in K_i$.
Since $\max_{u,v \in K_i} \codeg(u,v) =0$, we have that $a_i$ is not adjacent to $b_j$ if $i \neq j$. Hence, we can color the graph with three colors: $a_1,b_2$ are red, $a_2,b_3$ are yellow and $a_3,b_1$ are blue. Hence, $\chi(G)\le 3$, which gives the desired contradiction.

Now, we may assume without loss of generality that $|K_1|\geq 3$. 
Thus, $\max_{u,v \in K_i-T} \codeg(u,v)$ is well-defined. 
Moreover, as $(K_1,K_2,K_3)$ is a clique cover for $G$, $\codeg^{G-T}(u,v) \leq \codeg^G(u,v)-1$ for all $u,v \in K_i-T$, provided $|K_i|\ge 3$. Thus, $\max_{u,v \in K_i-T} \codeg(u,v)\leq \max_{u,v \in K_i} \codeg(u,v)-1$.
Then $G-T$ with the clique cover $(K_1-T, K_2-T, K_3-T)$ is also a counterexample as it satisfies \[\max_{u,v \in K_i-T} \codeg(u,v) +3 \leq \max_{u,v \in K_i} \codeg(u,v) -1 + 3 < \chi(G) - 1 \leq \chi(G-T).\] Hence, $G$ was not a minimal counterexample, a contradiction. 
\end{proof}

\begin{cor}\label{cor:three-cliqued}
    If $G$ is a claw-free graph with $\chi(\overline{G})\leq 3$, then $\chi(G)\leq \mcodeg(G)+3$.
\end{cor}


\subsection{The 2-joins}\label{ssec:$2$-joincases}
In this section, we consider the fifth outcome of Theorem~\ref{thm:structure-king-reed}, namely that $G$ admits one of several specific types of $2$-joins: a canonical interval $2$-join, an antihat $2$-join, a strange $2$-join, a pseudo-line $2$-join or a gear $2$-join. The case where $G$ admits a canonical interval $2$-join is already handled in Subsection~\ref{ssec:canonical_interval_2join}, so we will only consider the latter four.
The aim of this section is to show that no graph admitting one of those $2$-joins is $\mathcal{G}_{cf}$-critical for \MPdot.

The $2$-joins used in Theorem~\ref{thm:structure-king-reed} are defined as generalized $2$-joins $((X_1,Y_1),(X_2,Y_2))$ where $(G_2,X_2,Y_2)$ is one of several specific types of \say{strips}.
That is, if $(G_2, X_2, Y_2)$ is a 
\say{antihat strip}, \say{strange strip}, \say{pseudo-line strip}, or a \say{gear strip} then $((X_1, Y_1), (X_2, Y_2))$ is called an antihat $2$-join, a strange $2$-join, a pseudo-line $2$-join, or a gear $2$-join, respectively. 
A strip is a triple $(G, A, B)$ where $G$ is a claw-free graph, $A$ and $B$ are cliques of $G$ satisfying that for each $v \in A$ (resp. $B$) the neighbors of $v$ in $G \setminus A$ (resp. $G \setminus B$) form a clique.
For our purposes it is only important to remember that strips are triples $(G, A, B)$ where $G$ is a graph and $A,B$ are disjoint cliques of $G$.

In~\cite{king-reed-claw-free}, the class of antihat strips, class of strange strips, and class of gear strips are each defined by first defining smaller classes of strips that we call antihat ribbons, strange ribbons, and gear ribbons. 
(We postpone these definitions to the corresponding subsections.)
Then for each ribbon $H$ they define a set of allowed matchings $\mathcal{M}_H$ on $H$.
Finally, the class of antihat strips (respectively strange strips,  gear strips) is defined as the class of graphs $G$ of thickenings of $H$ along $M$ for any choice of antihat ribbon $H$ (respectively strange ribbon, gear ribbon) 
and matching $M \in \mathcal{M}_H$.
Note that by Lemma~\ref{lem:thickening-and-$2$-join}, 
if $G_2$ is a strict thickening of some other graph, then $G$ is not $\mathcal{G}_{cf}$-critical for \MPdot.
We call a $2$-join $((X_1, Y_2)$, $(X_2, Y_2))$ a \emph{ribbon $2$-join} if $(G_2, X_2, Y_2)$ is a ribbon.
Therefore, by the discussion above, we have the following.
\begin{obs}\label{obs:ribbons}
    Let $G$ be a claw-free graph that  is $\mathcal{G}_{cf}$-critical for \MPdot.
    Then every antihat $2$-join of $G$ is a \emph{ribbon} antihat $2$-join.
    Similarly, every strange $2$-join of $G$ is a ribbon strange $2$-join, and every gear $2$-join of $G$ is a ribbon gear $2$-join. 
\end{obs}
Consequently, for antihat, strange, and gear $2$-joins, it suffices to consider the corresponding ribbon graphs. In particular, we need not describe the allowed matchings $\mathcal M_H$ used in the original definitions.

The pseudo-line case is even simpler. We will observe that as a consequence of the definition of pseudo-line strips from~\cite{king-reed-claw-free}, whenever $(G_2, X_2, Y_2)$ is a pseudoline strip, the graph $G_2$ is a strict thickening of some other graph and therefore no $\mathcal{G}_{cf}$-critical graph for \MP can admit a pseudoline $2$-join by Lemma~\ref{lem:thickening-and-$2$-join}.

The first type of $2$-join from 
Theorem~\ref{thm:structure-king-reed}, namely the canonical interval $2$-join, is already discussed in Subsection~\ref{ssec:canonical_interval_2join}.
Each of the remaining four types of $2$-joins are defined and addressed separately in the rest of the subsection.

\subsubsection{Antihat 2-joins}
\begin{wrapfigure}{L}{0.16\textwidth}
\centering
  \begin{tikzpicture}[rotate=18, scale=0.8]
      \tikzstyle{vertex}=[circle, draw, fill=black,
                        inner sep=0pt, minimum width=4pt]
  \foreach \i in {0,...,5} {
    \node[vertex] (v\i) at ({cos(360/5*\i)}, {sin(360/5*\i)}) {};
  }

\node[vertex] (center) at (0,0) {};

  \foreach \i in {0,...,5} {
    \pgfmathtruncatemacro{\j}{mod(\i+1,5)}
    \draw (v\i) -- (v\j);   
    \draw (center) -- (v\i);
  } 
\end{tikzpicture}
\caption{$W_5$.}
\label{fig:w5}
\end{wrapfigure}

We need the following definition to define antihat $2$-joins. A wheel on 6 vertices, denoted $W_5$, is a cycle on $5$ vertices (i.e.~a $C_5$) and an additional vertex called the center adjacent to all of the other vertices (see Figure~\ref{fig:w5}).

We are now ready to define the class of antihat ribbons.
Let $H$ be a claw-free graph with vertex set $A \cup B \cup C$ where $A = \{a_1 ,\dots , a_k \}$, $B = \{b_1 ,\dots , b_k \}$, and $C = \{c_1 ,\dots , c_k \}$ are disjoint cliques. The adjacency between the cliques are as follows. For $1 \leq i, j \leq k$, $a_i$ and $b_j$ are adjacent if and only if $i = j$, and $a_i$ and $b_i$ are adjacent to $c_j$ if and only if $i\neq j$.
Let $X \subseteq V(H)$ such that $C$ contains at least two non-elements of $X$.
Then $(H \setminus X, A \setminus X, B \setminus X)$ is a strip.
If $H \setminus X$ contains a $W_5$, then we call $(G_2, X_2, Y_2) = (H \setminus X, A \setminus X, B \setminus X)$ an \emph{antihat ribbon}.

Recall that if $G$ admits a generalized $2$-join such that $(G_2 , X_2 , Y_2 )$ is an antihat ribbon, then we say that $((X_1 , Y_1 ), (X_2 , Y_2 ))$ is a \emph{ribbon antihat $2$-join}.

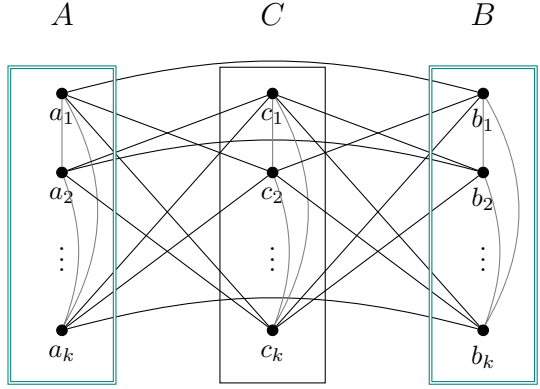
\begin{figure}[ht!]
    \centering
    \input{figures/antihat_strip3}
    \caption{Example of an antihat ribbon with $X = \emptyset$.}
\end{figure}

\begin{lem}\label{antihat critical}
    If $G$ admits an antihat $2$-join then $G$ is not  $\mathcal{G}_{cf}$-critical for \MPdot.
\end{lem}
\begin{proof}
    Suppose that $G$ admits an antihat $2$-join and that $G$ is $\mathcal{G}_{cf}$-critical for \MPdot. 
    Then by Observation~\ref{obs:ribbons}, $G$ admits a generalized $2$-join $((X_1, Y_1)(X_2, Y_2))$ such that $(G_2, X_2, Y_2)$ is an antihat ribbon.

    Since $G$ is $\mathcal{G}_{cf}$-critical for \MPdot, there exists $\varphi$, a coloring of $G_1$ with at most $\mcodeg(G_1)+3\leq \mcodeg(G)+3$ colors. By definition of antihat ribbon, $G_2$ contains a copy of $W_5$. We denote by $W \cup \{w\} \subseteq V(G_2)$ the vertex set of a $W_5$, where $w$ is the center of the wheel and $W$ is the vertex set of the five-cycle. 
    We exploit this copy of $W_5$ to extend $\varphi$ to $G_2$. 
     
    \begin{claim}\label{claim:sameindex}
        For some choice of index $i$, both of $a_i, c_i$ belong to $G_2$. 
    \end{claim}
    
    \begin{proofclaim}     Let us analyze how $W$ may embed in $G_2$ based on the definition of antihat ribbon. 
    Let $A', B', C'$ denote $A \setminus X$, $B \setminus X$, $C \setminus X$, respectively, from the definition of antihat ribbon.
    First, observe that since $W$ induces a cycle, at most two vertices of each clique $A',B',C'$ can be part of $W$. Since $W$ contains 5 vertices, it contains at least one vertex from each of $A',B'$ and $C'$.
    Note that by definition, for any two distinct $c_i, c_j \in C'$, every vertex in $G_2$ (including $c_i$ and $c_j$) is adjacent to at least one of $c_i, c_j$.
    Hence, $W$ contains exactly one vertex of $c_i \in C'$.
    Moreover, the two non-neighbors $a_i, b_i$ of $c_i$ in $A \cup B$ must be in $W$ as well.
    \end{proofclaim}

    We conclude by describing, for a contradiction, how to extend $\varphi$, the $(\mcodeg(G)+3)$-coloring of $G_1$, to a coloring of $G$.
    We repeat some key facts from the definition of $2$-join. 
    Recall that $X_1 \cup A'$ and $Y_1 \cup B'$ are both cliques and vertices of $A'$ have no neighbors in $G_1 \setminus X_1$ and vertices of $B'$ have no neighbors in $G_1 \setminus Y_1$.

    First, we color the vertex $a_i$ and suppose $\phi(a_i)=\pi$. Then we also assign $\phi(c_i)=\pi$, this is possible as $a_ic_i \not\in E(G)$.

    Each $b \in B'$ has at most $\omega(G) -1$ neighbors in the graph on $V(G_1) \cup B'$.
    Let $\pi$ denote a color used in $\varphi$.
    Since $\mcodeg(G)+3\geq \omega(G)+1$, we can extend $\varphi$ to $B'$ without introducing any more colors or using the color $\pi$.
    Since each $a' \in A'$ has at most $\omega(G)$ neighbors in the graph on $V(G') \cup A' \cup B'$, we can extend $\varphi$ to $A'$ without introducing any new colors. 
    
    Finally, greedily color the  $C\setminus (X \cup \{c_i\})$. We analyze why this does not introduce more colors.
    Note that for each $c_\ell \in C'$ with $\ell \neq i$, every neighbor of $c_\ell$ besides possibly $a_i, b_i, c_i$ is also a neighbor of $c_i$.
    So in particular, $\deg(c_\ell)\leq \mcodeg(G_2)+3\leq \mcodeg(G)+3$.
    By Claim~\ref{claim:sameindex} and the definition of an antihat ribbon, $a_i, c_i \in N(c_\ell)$.
    As $a_i$ and $c_i$ are neighbors of $c_\ell$ and $\varphi(a_i)=\varphi(c_i)$, this implies that there is at least one color available for $c_\ell$ with $\ell\neq i$.

    So the graph $G$ is $(\mcodeg(G)+3)$-colorable, implying that it is not $\mathcal{G}_{cf}$-critical for \MPdot, a contradiction.
\end{proof}

\subsubsection{Strange 2-joins}

Let $G_2$ be a claw-free graph on cliques $A = \{a_1 , a_2 \}$, $B = \{b_1 , b_2 , b_3 \}$, and $C = \{c_1 , c_2 \}$ with adjacency as follows: $a_1 , b_1$ are adjacent; $c_1$ is adjacent to $a_2, b_2$ and $b_3$ ; $c_2$ is adjacent to $a_1 , a_2 , b_1$, and $b_2$. All other pairs are nonadjacent. 
Then $(G_2, A, B)$ is a strip, and we call it a \emph{strange ribbon}.
Recall that if $G$ admits a generalized $2$-join such that $(G_2 , X_2 , Y_2 )$ is a strange ribbon, then we say that $((X_1 , Y_1 ), (X_2 , Y_2 ))$ is a \emph{ribbon strange $2$-join}.

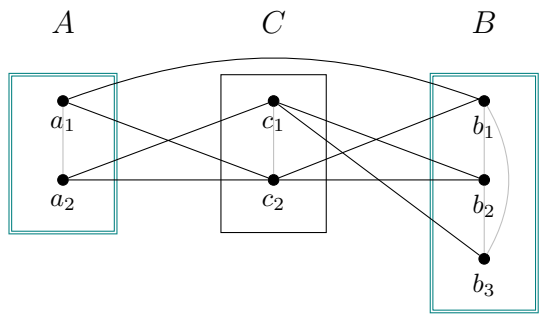
\begin{figure}[ht!]
    \centering
    \input{figures/strange_strip}
   \caption{The strange ribbon.}
    \label{fig:strangestrip}
\end{figure}

\begin{lem}\label{Lem: strange 2 join}
    If $G$ admits a strange $2$-join, then $G$ is not $\mathcal{G}_{cf}$-critical for \MPdot.
\end{lem}
\begin{proof}
Suppose that $G$ admits a strange 2-join and is $\mathcal G_{cf}$-critical for \MPdot. Then by Observation~\ref{obs:ribbons}, there is a strange ribbon $(G_2, X_2, Y_2)$ such that $((X_1, Y_1), (X_2,Y_2))$ is a strange $2$-join of $G$.

Since $G$ is $\mathcal{G}_{cf}$-critical, there exists a coloring $\varphi$ of $G_1$ with at most $\mcodeg(G_1)+3\leq \mcodeg (G) +3$ colors. We show that $\varphi$ can be extended to a coloring of $G_2$.

First, extend $\varphi$ to $A\cup B$ by observing that the only edges between $G_1$ and $G_2$ lie in the cliques $X \cup A$ and $Y \cup B$. 
From the definition of strange ribbon, we observe that 
$N_G(a_1) \setminus \{a_2, b_1\} = N_G(a_2) \setminus \{a_1, c_1\}$ 
So $\deg_G(a_1)=\deg_G(a_2)= \codeg^G(a_1,a_2)+2\leq \mcodeg(G) +2$.
Similarly, $N_G(b_1) \setminus \{b_2, a_1\} =N_G(b_2) \setminus \{b_1, c_1\}$. Thus,
$\deg(b_3) \leq \deg_G(b_1)=\deg_G(b_2) = \codeg^G(b_1,b_2)+2\leq \mcodeg (G)+2$. 
So $\varphi$ can be extended to $A \cup B$ greedily using at most $\mcodeg(G) + 3$ colors. 

Finally, extend $\varphi$ to $C$ by noticing that the vertices of $C$ only have neighbors of within $G_2$. Thus, $\deg_G(c_1)=4$ and $\deg_G(c_2)=5$. Moreover, since $\codeg^G(c_1,c_2)=|\{a_2,b_2\}|=2$, $\varphi$ can use at least five colors. Therefore, it is possible to first extend $\varphi$ to $c_2$, 
as at this stage $c_2$ only has four neighbors that have already received colors.
Since $\deg(c_1) = 4$, there will still be a color available for $c_1$.
Hence, the graph $G$ is then $(\mcodeg(G)+3)$-colorable, a contradiction with its $\mathcal{G}_{cf}$-criticality.
\end{proof}

\subsubsection{Pseudo-line 2-joins}
We do not need the full definition of pseudo-line $2$-joins or pseudo-line ribbons to eliminate this case.
As a consequence of the definition of King and Reed~\cite{king-reed-claw-free}, every pseudo-line strip $(G, A,B)$ satisfies both of the following:
\begin{enumerate}
    \item $G$ contains $W_5$, and
    \item there is a graph $J$ such that $G$ is the thickening of the line graph of $J$. 
\end{enumerate}
Since $W_5$ is not contained in any line graph (see e.g.~\cite{linegraphs}), we have the following.
\begin{obs}
    If $(G, A, B)$ is a pseudo-line strip then $G$ is the strict thickening of some other graph.
\end{obs}
Hence, by Lemma~\ref{lem:thickening-and-$2$-join}, the following holds.
\begin{obs}
    If a graph $G$ admits a pseudo-line $2$-join then $G$ is not $\mathcal{G}_{cf}$-critical for \MPdot.
\end{obs}

\subsubsection{Gear 2-joins}
The definition of gear $2$-joins was introduced by Galluccio et al.~\cite{galluccio2008gear} and restated by King and Reed~\cite{king-reed-claw-free}.

Let $H$ be a claw-free graph on vertices $\{v_1 ,\dots , v_{10} \}$ with adjacency as follows. The vertices $v_1 ,\dots , v_6$ form an induced cycle of length $6$, and $v_7$ is complete to $\{v_1,v_2,v_3,v_6\}$; $v_8$ is complete to $\{v_3,v_4,v_5,v_6,v_7\}$; $v_9$ is complete to $\{v_1, v_3, v_4, v_6,v_7,v_8\}$; $v_{10}$ is complete to $\{v_2,v_3,v_5,v_6,v_7,v_8\}$.
    There are no other edges in $H$. 
For any $X \subseteq \{v_9, v_{10}\}$ the triple $(H\setminus X, \{v_1, v_2\}, \{v_4, v_5\})$ is a gear ribbon.
Recall that if $G$ admits a generalized $2$-join such that $(G_2 , X_2 , Y_2 )$ is a gear ribbon, then we say that $((X_1 , Y_1 ), (X_2 , Y_2 ))$ is a \emph{ribbon gear $2$-join}.

\begin{figure}[ht!]
    \centering
    \input{figures/gear_strip2}
    \caption{The gear ribbon with $X=\emptyset$. The shapes at each vertex represent their colors from the proof of Lemma~\ref{lemma:gear2join}.}
    \label{fig:gearstrip}
\end{figure}
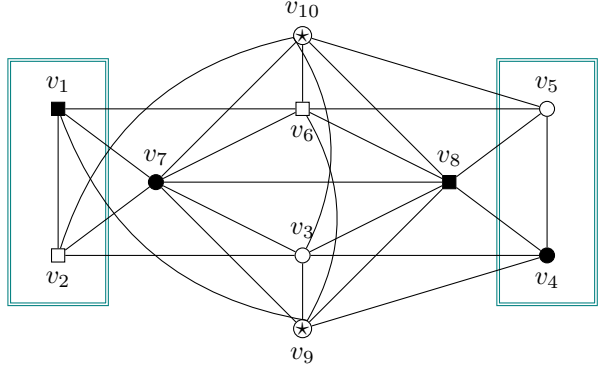

\begin{lem}\label{lemma:gear2join}
    Any graph admitting a gear $2$-join is not  $\mathcal{G}_{cf}$-critical for \MPdot.
\end{lem}
\begin{proof}
Suppose that $G$ admits a gear $2$-join and is $\mathcal G_{cf}$-critical for \MPdot. Then by Observation~\ref{obs:ribbons}, $G$ admits a generalized $2$-join $((X_1, Y_1), (X_2,Y_2))$ where $(G_2,X_2,Y_2)$ is a gear ribbon.
Then using the notation introduced in the definition of gear ribbon, for some $X \subseteq \{v_9, v_{10}\}$, $(G_2, X_2, Y_2) = (H\setminus X, \{v_1, v_2\}, \{v_4, v_5\})$.

Since $G$ is $\mathcal{G}_{cf}$-critical, there exists a coloring $\varphi$ of $G_1$ with $\mcodeg(G)+3$ colors. Let us show that it is possible to extend $\varphi$ to $G_2$.

Note that both $v_1, v_2$ and $v_4, v_5$ have a common neighbor in $G_2$.
By definition of $2$-join, the set of neighbors of $v_1, v_2$ in $G_1$ is $X_1$ and the set of neighbors of $v_4, v_5$ is $X_2$.
Each of $v_1, v_2, v_4, v_5$ is adjacent to $\mcodeg(v_1,v_2) -1$ respectively $\mcodeg(v_4,v_5)-1$ colored vertices.
Hence, since we color with $\mcodeg(G) + 3$, vertices there are at least four colors available at each of $v_1, v_2, v_4, v_5$.
Hence, we can iteratively extend $\varphi$ to $v_1, v_2, v_4, v_5$ in such a way that they all receive distinct colors.

We now extend the coloring to the remaining vertices of $v_1, v_2, \dots, v_8$ as follows: $\varphi(v_3)=\varphi(v_5)$, $\varphi(v_6)=\varphi(v_2)$, $\varphi(v_7)=\varphi(v_4)$ and $\varphi(v_8)=\varphi(v_1)$.
(See Figure~\ref{fig:gearstrip} for an illustration.)
Note that if $X =\{v_9, v_{10}\}$, then we have fully colored $G$ with $\mcodeg(G) + 3$ colors for a contradiction.

So we may assume that one of $v_9, v_{10} \in V(G_2)$.
In either case, $\mcodeg(G) \geq 4$ as $v_9, v_8$ have four common neighbors in $H \setminus \{v_{10}\}$ and $v_{10}, v_7$ have four common neighbors in $H \setminus \{v_{9}\}$.
So in particular, $\mcodeg(G) + 3 \geq 7$.
However, $\deg_H(v_9) = \deg_H(v_{10}) \leq 6$, so we can greedily extend $\varphi$ to $v_9, v_{10}$ if necessary.
Therefore, $G$ is $(\mcodeg(G) + 3)$-colorable, a contradiction.
\end{proof}

\subsection{Icosahedral thickenings}
\label{sec:icosahedron}
By Lemma~\ref{lem:thickenings}, we know that if $G$ is a strict thickening of a graph $G'$ then $G$ is not $\mathcal{G}_{cf}$-critical for \MPdot. Hence, we only need to consider the case where $G$ is an induced subgraph of the icosahedral graph.

\begin{lem}\label{lem:icosahedral}
    If $H$ is an induced subgraph of the icosahedral graph, then $\chi (H) \leq \mcodeg (H) +3$.
\end{lem}

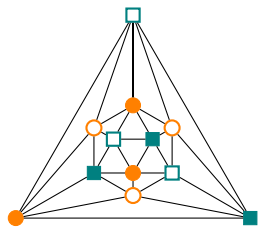
\begin{figure}[ht!]
    \centering
    \input{figures/icosahedron}
    \caption{Icosahedral graph with a $4$-coloring.}
    \label{fig:icosahedron}
\end{figure}

\begin{proof}
    First, note that the icosahedron is $4$-colorable as shown in Figure~\ref{fig:icosahedron}. Therefore, if $\mcodeg (H)\geq 1$, then $\chi(H)\leq 4 \leq \mcodeg +3$, as desired.

    Suppose now that $\mcodeg(H)= 0$. Then $H$ is triangle-free. 
    Let $u$ be a vertex of the icosahedron and $v\in N(u)$. Observe that in the icosahedron, every vertex is contained in five triangles and any edge is contained in exactly two triangles. Therefore, if $u$ is a vertex in a subgraph of the icosahedron, removing a vertex in its neighborhood will delete at most two triangles. So for any vertex $w\in H$, at least three neighbors must have been removed from the icosahedron to delete the five triangles it is contained in. Thus, $|N(w)|\leq 5 - 3$ for any vertex $w\in H$, and we conclude that $\chi (H)\leq 3 = \mcodeg (H) +3$.
\end{proof}

\section{Adaptation to edge-codegree}\label{sec:adaptation}

In this section, we explain what modifications are necessary for the following strengthened form of Theorem~\ref{thm:maincf}. 
Recall that the maximum edge-codegree of a given graph $G$, denoted by $\Delta_e(G)$, is the maximum size of a collection of triangles in $G$ all containing some common edge.
\maincfedge*

\noindent
To prove Theorem~\ref{thm:maincfedge}, the overall strategy is unchanged: we go through the graph types in the structural description of claw-free graphs given in Theorem~\ref{thm:structure-king-reed} in the same order as for Theorem~\ref{thm:maincf}, except this time we show that none of them contains a critical graph for the edge-codegree version of \MPdot.

\begin{propertye} 
    $\chi(G)> \Delta_e (G) + 3$. 
\end{propertye}

First, note that $\Delta_e(G)$ also satisfies $\Delta_e(G)+2 \geq \omega(G)$. Hence, all lower bounds on the codegree that come from cliques remain valid.
As we sometimes work in the complement, we also use $\overline{\Delta_e}(G)$ to denote $\Delta_e({\overline{G}})$. 

\subsection*{Defuzzification}
We remark that no vertex-critical graph for \MPe contains a clique cutset so Lemma~\ref{lem:no-clique-cutset} holds for \MPe as well.

However, a vertex-critical graph for \MPe can contain a vertex dominating some other vertex, as long as these vertices are not adjacent. Hence, Lemma~\ref{lem:nodominatingvertices} does not hold and we need the following adapted form.

\begin{lem}[cf.~Lemma~\ref{lem:nodominatingvertices}]  If $G$ is vertex-critical for \MPe and $x$ and $y$ are adjacent, then $x$ does not dominate $y$.
\end{lem}
\begin{proof}
    Suppose $G$ is vertex-critical for \MPe and there exist adjacent vertices $x,y$ with $N(y) \subseteq N[x]$.
    Then $\deg(y) \leq \codeg(x,y) +1 \leq \Delta_e(G) + 1$. In particular, any coloring of $G-y$ with at least $\Delta_e(G)+2$ 
    colors can be extended to $G$, by using the least available color for $y$. By vertex-criticality of $G$, it follows that $\chi(G)\le \Delta_e(G-y)+2\leq \Delta_e(G)+3$, a contradiction.
\end{proof}

Lemma~\ref{lem:nodominatingvertices} is used in the proof of Lemma~\ref{lem:trivial-homog-pair-of-cliques}, in which a vertex dominates another to reach a contradiction. However, these vertices are in the same clique, and hence the result remains true for \MPedot. Consequently, the main result concerning homogeneous pairs of cliques still holds.
\begin{lem}[cf.~Lemma~\ref{lem:no-homog-pair-of-cliques}]\label{lem:no-homog-pair-of-cliques edge}
If $G$ is $\mathcal{G}_{cf}$- or $\mathcal{G}_{ql}$-critical for \MPedot, then $G$ has no pair of homogeneous cliques.
\end{lem}

Later, Lemma~\ref{lem:nodominatingvertices} is used to obtain Lemma~\ref{lem:thickenings} from Lemmas~\ref{lem:thickening-means-homog-or-twins} and~\ref{lem:no-homog-pair-of-cliques}. This relies on the existence of twins, which are vertices dominating each other, and a closer look at the proof of Lemma~\ref{lem:thickening-means-homog-or-twins} shows that those twins are adjacent. Therefore, the analogous result holds for \MPedot.

\begin{cor}[cf.~Lemma~\ref{lem:thickenings}]\label{lem:thickenings edge}
  If $G$ is $\mathcal{G}_{cf}$-critical or $\mathcal{G}_{ql}$-critical for \MPedot, then $G$ is not a strict thickening of some graph $G_0$.
\end{cor}

Last, this cascades to the proof for generalized $2$-joins using the previous adapted lemmas. 
\begin{lem}[cf.~Lemma~\ref{lem:thickening-and-$2$-join}]\label{ribbon 2 join edge}
Suppose $G$ admits a generalized $2$-join $((X_1,Y_1),(X_2,Y_2))$ such that $G_2$ is a strict thickening of some graph $H_2$ and such that $X_2 = I(A)$ and $Y_2 =I(B)$ where $A,B$ are cliques of $H_2$. Then $G$ is neither $\mathcal{G}_{cf}$-critical nor $\mathcal{G}_{ql}$-critical for \MPedot.
\end{lem}

\subsection*{Quasi-line graphs}
We next describe how to prove the edge-codegree strengthening for quasi-line graphs.
\begin{thm}[cf.~Theorem~\ref{thm:mainquasiline}]\label{delta_e quasiline} If $G$ is a quasi-line graph, then $\chi(G) \leq \Delta_e(G)+3$.
\end{thm}

Since $\Delta_e(L(G))+2\geq \omega(L(G)) \geq \Delta(G)$, Vizing's theorem gives the bound for line graphs of simple  graphs. Moreover, using the same arguments as in Lemma~\ref{lem:linegraph_multigraphs}, we can extend the result to line graphs of multigraphs. Also, the proof of Lemma~\ref{lem: degenerate circular interval} directly extends to the edge-codegree setting as we only looked at the codegree of vertices that lie in a common interval. The same holds for the proof of Lemma~\ref{lem:canonical interval 2 join}, where we prove that graphs admitting a canonical interval $2$-join are not critical for \MPedot. This yields the proof of Theorem~\ref{delta_e quasiline}. 

\subsection*{Claw-free graphs}
We now extend this to claw-free graphs. By Theorem~\ref{thm:structure-king-reed} and Lemma~\ref{lem:thickenings edge}, we need to consider antiprismatic graphs, three-cliqued graphs, several kinds of graphs admitting a ribbon $2$-join, and icosahedral graphs.

\subsubsection*{Antiprismatic graphs}
We use the graphs where $\alpha(G)=2$ as a base case for extending our results to antiprismatic graphs. However, Lemma~\ref{lem: bounding anti codeg} does not hold if we replace $\Delta_2$ with $\Delta_e$. If $H=\overline{G}$, for example, equals $C_5$, then $\Delta_e(H)=0$, whereas $\frac{|V(H)|}{2}-2=0.5$. Hence, we need to relax the bound slightly and show that $\overline{\chi}(H)\leq \overline{\Delta_e}(H)+3$. 

\begin{lem}[cf.~Lemma~\ref{lem: bounding anti codeg}]\label{lem: bounding anti edge codeg} If $H$ is a triangle-free graph, then $\overline{\Delta_{e}}(H) \geq \frac{|V(H)|}{2}- 3$. 
\end{lem}

\begin{proof} Let $n=|V(H)|$. The statement is trivial if $n \leq 6$.
    Suppose for a contradiction that every pair of vertices $v,w \in V(H)$ such that $vw \not \in E(H)$ satisfies $\acodeg(v,w) \leq \frac{n}{2}-\frac{7}{2}$. Then for all $vw \not\in E(H)$ 
     \[|N(v) \cup N(w)| \geq
    (n-2)-\left(\frac{n}{2}-\frac{7}{2}\right)=\frac{n}{2}+\frac{3}{2}.\]

    We will first prove the result for all $H$ that are also $C_4$-free, i.e.~all $H$ with girth at least 5.
    If every vertex $v \in H$ has degree at most 2, then for all $vw \not \in E(H)$, we have \[\acodeg(v,w)\geq (n-2)-|N(v)|-|N(w)| \geq n -2 -2 -2=n-6 > \frac{n}{2} -3\] as $n > 6$.
    Hence, suppose there exists a vertex $u \in H$ of degree at least 3.
    Let $v,w,x$ be three neighbors of $u$.
    Since $H$ is triangle-free, these three vertices are an independent set. Let $y \neq u,v,w,x$ be any other vertex in $H$. 
    Then, since $H$ is $C_4$-free, $y$ is adjacent to at most one of $v,w,x$. 
    Suppose without loss of generality that $y$ is not adjacent to both $w$ and $x$.
    Then $|N(v) \cup N(w)| + |N(x) \cup N(y)| \geq n+3$. 
    
    Moreover, notice that neither $w$ nor $x$ belongs to either set. Hence, at least five vertices appear in both the sets $N(u) \cup N(w)$ and $N(x) \cup N(y)$.
    Since each of these five vertices must be adjacent to $u$ or $w$ and to $x$ or $y$, by the pigeonhole principle some two vertices must have two common neighbors in $v,w,x,y$.
    Hence, $H$ contains $C_4$ as a subgraph, a contradiction.
       
    Let us now prove the result if $H$ contains a $C_4$. Let $v_1,v_2,v_3,v_4$ be such that $v_iv_{i+1} \in E(H)$ for all $i=1,\ldots, 4$ (where $v_5=v_1$). Then, since $H$ is triangle-free, we have that $N(v_i) \cap N(v_{i+1}) = \emptyset$. We reach a contradiction as follows
    \begin{align*} n &\geq |N(v_1)\cup N(v_2) \cup N(v_3) \cup N(v_4)|= |N(v_1) \cup N(v_3)| + |N(v_2)\cup N(v_4)|\\ &\geq \frac{n}{2}+\frac{3}{2} + \frac{n}{2}+\frac{3}{2} = n+3. \qedhere\end{align*}
\end{proof}

Since multiple statements depend on Lemma~\ref{lem: bounding anti codeg}, we need to adapt those results according to the relaxed bound. However, a careful look at their proofs shows that if we replace $\Delta_e$ with $\mcodeg$ and replacing the constant 2 with the constant 3.

\begin{thm}[cf.~Theorem~\ref{thm:alpha2_chiBounded}]
    If $G$ is a graph such that $\alpha(G)\le2$, then $\chi(G) \leq \Delta_e(G)+3$.
    Equivalently, if $H$ is triangle-free, then $\overline{\chi}(H) \leq \overline{\Delta_{e}}(H)+3$.
\end{thm}



\begin{lem}[cf.~Lemma~\ref{lem: prismatic}]
    If $H$ is a prismatic graph, then $\overline{\chi}(H) \leq \overline{\Delta_e}(H) + 3$.
\end{lem}

\begin{cor}[cf.~Corollary~\ref{cor: antiprismatic}]
    If $G$ is an antiprismatic graph, then $\chi(G) \leq \Delta_e(G)+3$.
\end{cor}

\subsubsection*{Three-cliqued graphs}
For the three-cliqued graphs, we only considered the codegree of vertices that have the same color in the complement graphs. Since in the three-cliqued graphs these vertices have to be adjacent, the proofs for the statements with $\Delta_e$ instead of $\Delta_2$ are unchanged and the following holds.
\begin{cor}[cf.~Corollary~\ref{cor:three-cliqued}]
    If $G$ is a claw-free graph with $\chi(\overline{G})\leq 3$, then $\chi(G) \leq \Delta_e(G)+3$.
\end{cor}

\subsubsection*{The 2-joins}
Next, we look at the graphs that admit a $2$-join $((X_1,Y_1),(X_2,Y_2))$, where $(G_2,X_2,Y_2)$ is one of the specified types of strips. We have already seen that $G$ is not critical for \MPe if $G$ admits a canonical interval $2$-join. 
Lemma~\ref{ribbon 2 join edge} implies that if $G$ is critical for \MPe then $G_2$ is a ribbon.

Suppose $G$ is a ribbon antihat $2$-join. In the proof of Lemma~\ref{antihat critical}, we only consider codegrees of vertices that lie both in $A$, both in $B$ or both in $C$. Since $A$, $B$ and $C$ are cliques, $G$ is not critical for \MPe if $G$  is a ribbon antihat $2$-join. Also in the proof of Lemma~\ref{Lem: strange 2 join} we only look at codegrees of vertices of the same type. Hence, we can use this proof to show that $G$ is not critical for \MPe if $G$  is a ribbon strange $2$-join.  Moreover, if $G$ is a pseudoline $2$-join, then $G_2$ is a thickening of some other graph. Hence, by Lemma~\ref{ribbon 2 join edge}, $G$ is not critical for \MPedot. Finally, as in the proof of Lemma~\ref{lemma:gear2join}, we only consider codegrees of adjacent vertices. So using this proof, we can show that $G$ is not critical for \MPe if $G$  is a ribbon gear $2$-join.

\subsubsection*{Icosahedral graphs}
For the icosahedral graphs, we observe that the proof of Lemma~\ref{lem:icosahedral} also works for edge-codegree as $\Delta_e(H)=0$ implies that $H$ is triangle-free.

\bigskip
This completes our description of how to adapt our proof of  Theorem~\ref{thm:maincf} for Theorem~\ref{thm:maincfedge}.

\section{Concluding remarks}\label{sec:conclusion}

Shortly after stating it, we hinted at a possible sharpening of Conjecture~\ref{conj:Vu}.

The following is stronger in two senses: first, we tighten the bound $\mcodeg(G)+\varepsilon_2\mcodeg(G)$ in the larger codegree regime to align with Vizing's theorem; second, we conjecture a bound in the smaller codegree regime too.

\begin{conj}\label{conj:Vusharper}
For some $C>0$, $\chi(G) \le \max\{C\Delta(G)/\log\Delta(G),\ \mcodeg(G)+3\}$.
\end{conj}
\noindent
If true, the first term in the maximization would be sharp up to the choice of $C$ due to random regular graphs, the other term due to certain line graphs. 
We might dub this formulation ---as well as its strengthening to bound $\chi_\ell(G)$ instead of $\chi(G)$--- the {\em Strong Vu Conjecture}.
The hypothetical transition between the two regimes represents an inherent mixture of probabilistic and structural graph theoretic considerations.
We note that Conjecture~\ref{conj:Vusharper} in the $\mcodeg(G) = O(\Delta(G)/\log\Delta(G))$ regime would, if true, strengthen results in~\cite{BDMW25+,CJMS23+}.
Moreover, by a random partitioning argument (see~\cite[Sec.~3]{BDMW25+}), a bound of $\chi(G)\le (1+o(1))\Delta(G)/\log \Delta(G)$ provided $\mcodeg(G) \le (1+o(1))\Delta(G)/\log \Delta(G)$ would suffice to prove the conjecture (up to a $(1+o(1))$ factor in the second term of the maximization); we thank Noga Alon for pointing out to us this implication.
Besides Conjecture~\ref{conj:Vu} and Conjecture~\ref{conj:Vusharper}, there are several other natural directions for further study.

First, let us call a claw-free graph $G$ {\em Vu-critical} if $\chi(G) = \mcodeg(G)+3$ but $\chi(H) < \mcodeg(H)+3$ for every proper induced subgraph $H$ of $G$. What is the class of all Vu-critical graphs? Does it contain any graph that is not the line graph of a graph of class II?

Second, are there other hereditary graph classes within which we can bound the chromatic number of $G$ by $(1+o(1))\mcodeg(G)$ ---apart from those where we derive the result as corollary to some analogous bound in terms of $\omega(G)$?
Let  $\mathcal G$ be a hereditary class.
Recall that if we have established some bound of the form $\chi(G) \le f(\omega(G))$ for all $G\in{\mathcal G}$, then we immediately also have a bound of the form $\chi(G)\le f(\mcodeg(G)+2)$ for all $G\in{\mathcal G}$.
We can thus disregard $\mathcal G$ having only graphs satisfying $\chi(G) \le \omega(G)+O(1)$ (e.g.~\cite{esperet2016coloring-regions}).
As an aside, one might be lured into thinking of $\mathcal G$ being all $K_{1,r}$-free graphs or, more restrictively, graphs all of whose neighborhoods are the union of $r-1$ cliques, $r \ge 4$ as the most natural next step, but we caution that this includes line graphs of linear $(r-1)$-uniform hypergraphs as a special case~\cite{PiSp89}.
We surmise that the class of all bull-free graphs could be a good next target, particularly given their thorough investigation from a structural standpoint~\cite{Chudnovsky2012a,Chudnovsky2012, ErdosHajnalBulls}.
Irrespective of resolving Conjecture~\ref{conj:Vu} for $\mathcal G$, it seems interesting to pursue bounds for the optimal $f_{\mathcal G}$ (if it exists) such that $\chi(G) \le f_{\mathcal G}(\mcodeg(G))$ for all $G \in {\mathcal G}$.
This would provide a natural parallel to the study of $\chi$-boundedness~\cite{Gya87}. 

Third, it would be natural to extend Theorem~\ref{thm:maincf} to list coloring, in the spirit of Vu's original formulation of Conjecture~\ref{conj:Vu}.
\begin{conj}\label{conj:cflist}
If $G$ is claw-free, then $\chi_\ell(G) \le \mcodeg(G)+3$. 
\end{conj}
\noindent
This is essentially a generalization of the so-called Weak List Coloring Conjecture, which states for any simple graph that the list chromatic index is bounded by one plus the maximum degree.
Due to Theorem~\ref{thm:maincf}, Conjecture~\ref{conj:cflist} is a special case of a conjecture of Gravier and Maffray~\cite{GrMa04}, which asserts that the list chromatic number coincides with the chromatic number in any claw-free graph.

%

In the above suggested lines of further study, if we substituted $\Delta_e$ (the codegree maximized over the endpoints of edges) in the place of $\mcodeg$, then we would obtain a few more interesting problems. In particular, the analogues of Conjecture~\ref{conj:Vu}, Conjecture~\ref{conj:Vusharper}, and Conjecture~\ref{conj:cflist} with  $\Delta_e$ in the place of $\mcodeg$ may be worth investigation.

\paragraph{Acknowledgements.} 
This work was initiated during a month-long visit by ER to the Korteweg--de Vries Institute for Mathematics (KdVI) at the University of Amsterdam. 
We thank the organizers of the KdVI cookie time. 
RK is grateful to Lotte Paterek, whose master's thesis survey on the List Coloring Conjecture was a timely reminder of the edge-coloring problems and results that inspired this project. We thank the reviewers for their time and for suggesting a stronger formulation of Lemma~\ref{lem: degenerate circular interval}, which also led to a simpler proof.

This project began when LC was affiliated with the University of Amsterdam. LC was partially supported by the Gravitation Programme NETWORKS (024.002.003) of the Dutch Ministry of Education, Culture and Science (OCW) and the European Union under the EU Horizon 2020 Research and Innovation Program under the Marie Sk{\l}odowska-Curie Grant Agreement (101034253). RK was partially supported by the Dutch Research Council (NWO) grant OCENW.M20.009 and the Gravitation Programme NETWORKS (024.002.003) of the OCW. ER was partially supported by the Belgian National Fund for Scientific Research (FNRS).

\paragraph{Emails.} \protect\href{mailto:l.j.cook@uu.nl}{\protect\nolinkurl{l.j.cook@uu.nl}}, \protect\href{mailto:r.kang@uva.nl}{\protect\nolinkurl{r.kang@uva.nl}}, \protect\href{mailto:eileen.robinson@ulb.be}{\protect\nolinkurl{eileen.robinson@ulb.be}}, \protect\href{mailto:g.c.zwaneveld@uva.nl}{\protect\nolinkurl{g.c.zwaneveld@uva.nl}}.

\paragraph{Open access statement.} For the purpose of open access,
a CC BY public copyright license is applied
to any Author Accepted Manuscript (AAM)
arising from this submission.

\bibliography{bibli}
\bibliographystyle{abbrv}

\appendix

\section{A better bound for circular interval graphs}\label{Section: Appendix better bound}
\begin{lem}\label{lem: shift graph}
    Given a positive integer $k$, let $G$ be a circular interval graph such that for every set of $k$ consecutive vertices on the underlying circle of $G$ there is precisely one interval of size $k$ that contains it. Then $\chi_{\ell}(G) \leq \Delta_e(G)+2$. 
\end{lem}
\begin{proof}
If $G$ is a complete graph on $n$ vertices, then $\Delta_e(G)=n-2$ and $\chi(G)=n$ as desired. 
Otherwise $k \leq \frac{|V(G)|}{2}$, and then $\Delta(G) = 2k-2$ and $\Delta_e(G) = 2k-4$. By the list version of Brooks' theorem~\cite{vizing1976vertex}, the only graphs such that $\chi_{\ell}(G) > \Delta(G)$ are odd cycles and complete graphs. Odd cycles satisfy $\Delta_e(G)= 1$ and $\chi(G) =3$. Hence, $\chi(G) \leq \Delta_e(G) +2$.
\end{proof}

For the next proof, we adopt some of the terminology for circular interval graphs that we introduced at the beginning of Subsection~\ref{sub:circular}.

\begin{lem}\label{lem:circular_interval_graphs_tight}
    If $G$ is a circular interval graph, then $\chi_{\ell}(G)\leq \Delta_e(G)+2$. 
\end{lem}
\begin{proof}
Let $C$ denote the circle and let $\mathcal I$ denote the set of intervals of $C$ corresponding to $G$. If for every vertex $v_i \in V(G)$, $v_i$ is the rightmost vertex of some longest interval in $\mathcal I$, then we conclude by Lemma \ref{lem: shift graph}. Otherwise, there are two consecutive vertices (along $C$), without loss of generality let us say $v_{n}, v_1 \in V(G)$, such that $v_1$ is the rightmost vertex of a longest interval in $I \in \mathcal{I}$ but $v_{n}$ is not the rightmost vertex of any longest interval in $I$. 
We order the vertices $v = v_1, v_2, \dots, v_n, v_1$ along the cycle $C$. We will use the notation $I_i^L:=I_{v_i}^L$ and $I_i^R:= I_{v_i}^R$ to improve readability. 

We consider the vertices of $G$ in the order starting at $v_1$ and moving along $C$ to the right. So we color first $v_1$, then $v_2$, then $v_3$ and so on. For every vertex $v_i$, we (try to) pick $\phi(v_i)$ to be the smallest color in their list that does not appear in their already colored neighborhood. 

For all vertices $v_i \not \in I_{1}^L$, the only neighbors of $v_i$ that are colored before $v_i$ are the vertices in $I_{i}^L$. Hence, $v_i$ has at most $I_i^L-1\leq \Delta_e(G)+1$ already colored neighbors, and so we can color all such $v_i$.

For a vertex $v_i \in I_1^L$, we denote by $A(v_i)$ the set of already colored neighbors of $v_i$ and by $R(v_i)$ the set of already colored neighbors of $v_i$ that lie to the right of $v_i$. In particular, $R(v_i)$ is of the form $\{v_1,\ldots,v_j\}$ for some $j$. The following claim follows from the fact that $|I^L_1|$ is assumed to be a maximum size interval.

\begin{claim}\label{precolor-divide}
For any $v_i \in I^L_1$
$|A(v_i)| = |R(v_i)| + |I_i^L|-1 \leq |R(v_i)| + |I_1^L|-1$.
\end{claim}

We bound $|A(v_i)|$ by examining $R(v_i)$ more closely.
Recall that by our choice of $v_1$, $|I_1^L| = \omega(G)$. 

\begin{claim}\label{precoloring-claim}
 For any $v_i \in I^L_1$,
 $|R(v_i)|\leq \Delta_e(G) +3- \omega(G)$.
\end{claim}

\begin{proofclaim}
Let $v_i \in I_1^L$. Then $v_1v_i \in E(G)$.
Note that if $R(v_i)$ is non-empty it must consist of the vertices corresponding to some interval starting at $v_1$ and traversing $C$ \emph{rightwards}.
So in particular all vertices in $R(v_i)$ besides $v_1$ are common neighbors of $v_1$ and $v_i$.
Note that since $v_1$ is the rightmost vertex of $I_1^L$, $R(v_i) \setminus \{v_1\}$ and $I_1^L \setminus \{v_1,v_i\}$ are disjoint.
So in particular $\codeg(v_1,v_i) \geq |I_1^L| + |R(v_i)| -3$.
So $|R(v_i)| \leq \codeg(v_1,v_i) - |I_1^L| +3$ and the first inequality of the claim follows. 
\end{proofclaim}

By combining Claims~\ref{precolor-divide} and~\ref{precoloring-claim} we obtain the following.

\begin{claim}\label{awclaim}
For every $v_i \in I_1^L$,
$|A(v_i)| \leq \Delta_e(G) +3-|I_1^L|+|I_i^L|-1 \leq \Delta_e(G)+2$.
If equality holds, then $|I_{i}^L| = \omega(G)$ and $|R(v_i)| =\Delta_e(G) +3-\omega(G)$, so that $R(v_i)=\{v_1,\ldots, v_{\Delta_e(G) +3-\omega(G)}\}$.
\end{claim}
%

We now prove that this procedure has at most one vertex that we cannot color greedily.

\begin{claim}\label{one-bad}
    In this procedure, there is at most one vertex $v_i$ that has at least $\Delta_e(G)+2$ neighbors that are colored before $v_i$. Moreover, $v_i$ has exactly $\Delta_e(G)+2$ neighbors that are colored before it.
\end{claim}

\begin{proofclaim}
    Suppose for a contradiction that $v_i$ and $v_j$ have this property where $i < j$.
    Then $v_i, v_j \in I_1^L$ and by Claim~\ref{awclaim}, $|I_1^L| = |I_i^L| = |I_j^L|=\omega(G)$, and $v_1,v_i, v_j$ are all pairwise adjacent and occur in this order when traversing rightwards along $C$ from $v_1$.
    Then, since $G$ is a circular interval graph, $R(v_i) \subseteq R(v_j)$.
    Since $v_i, v_j$ are both in $I_1^L$ and $|I_{i}^L| = |I_{j}^L| = \omega(G)$, it 
    follows that $v_{i-1}$ is adjacent to both $v_i$ and $v_j$.
    So in particular, the common neighbors of $v_i$ and $v_j$ include $I_1^L \setminus \{v_i, v_j\}$ and $R(v_j)$ and $\{v_{i-1}\}$.
    Since these sets are pairwise disjoint, we obtain that $\codeg(v_i, v_j) \geq  \Delta_e(G) +1$, a contradiction since $v_i, v_j$ are adjacent. 
\end{proofclaim}

We may assume that $v_i \in I_1^L$ is the unique uncolored vertex with $\Delta_e(G) + 2$ colored neighbors.
Let $v_{i-1}$ be the left neighbor of $v_i$. If $v_{i-1}$ has at most $\Delta_e(G)$ neighbors colored before it, we can switch the order of $v_{i-1}$ and $v_i$ so they both have at most $\Delta_e(G)+1$ forbidden colors when we color it. Moreover, since $v_{i-1}$ is the left neighbor of $v_i$, we have for every vertex $v_j \not\in \{v_{i-1},v_i\}$ that the set of already colored neighbors of $v_j$ is still the set $\{v_1,\ldots,v_{j-1}\}$.
So in particular, the arguments of the above claims still hold.
Hence, we can color all vertices as desired.

So assume that $v_{i-1}$ has $\Delta_e(G) +1$ neighbors colored before it.
Since $\codeg(v_{i-1},v_i) \leq \Delta_e(G)$, $v_{i-1}$ has at least one colored neighbor that is not a neighbor of $v_i$.
Hence, $|I_{i-1}^L| = |I_i^L| = |I_1^L|$ and $v_{i-1}$ is the rightmost vertex of $|I_{i-1}^L|$.
Moreover, $v_{i-1}$ has at least $\Delta_e(G)+1-(|I_{i-1}^L|-1)=\Delta_e(G)+2-|I_{1}^L|$ neighbors in $I^R_{i-1}$ that are colored before $v_{i-1}$. So $R(v_i) \supseteq \{v_1,\ldots, v_{\Delta_e(G)+2-|I_{1}^L|}\}$. Hence, $\codeg(v_{i-1},v_i) \geq (|I_i^L| -2) + (\Delta_e(G) +2-|I_1^L|) = \Delta_e(G)$. In particular, $v_{i-1}$ and $v_i$ have no other common neighbors.
This implies that there are no vertices in the interval $\{v_{i+1}, \ldots, v_n\}$, and so $v_i=v_n$. This contradicts the definition of $v_1$, as its left neighbor $v_n$ is the rightmost vertex of the interval $I_n^L$ of maximum length.
\end{proof}

\end{document}

%% file: figures/PetersenLinegraph.tex
\begin{tikzpicture}[scale=0.35, rotate=90]
    \tikzstyle{vertex}=[circle, draw, fill=black,
                        inner sep=0pt, minimum width=3pt]
    \tikzstyle{prev}=[rectangle, draw=orange, fill=orange,
                        inner sep=1.5pt, minimum width=3pt]

    \node[prev] (a1) at (0:2.5) {};
    \node[prev] (a2) at (72:2.5) {};
    \node[prev] (a3) at (144:2.5) {};
    \node[prev] (a4) at (216:2.5) {};
    \node[prev] (a5) at (288:2.5) {};
    \draw[color=orange] (a1)--(a3)--(a5)--(a2)--(a4)--(a1);

    \node[prev] (b1) at (0:5) {};
    \node[prev] (b2) at (72:5) {};
    \node[prev] (b3) at (144:5) {};
    \node[prev] (b4) at (216:5) {};
    \node[prev] (b5) at (288:5) {};
    \draw[color=orange](b1)--(b2)--(b3)--(b4)--(b5)--(b1);
    \draw[color=orange](a2)--(b2)--(b1)--(a1);
    \draw[color=orange](a3)--(b3)--(b4)--(a4);
    \draw[color=orange](a5)--(b5);
    
    \node[vertex] (w1) at (0:0.75) {};
    \node[vertex] (w2) at (72:0.75) {};
    \node[vertex] (w3) at (144:0.75) {};
    \node[vertex] (w4) at (216:0.75) {};
    \node[vertex] (w5) at (288:0.75) {};
    \draw (w1)--(w3)--(w5)--(w2)--(w4)--(w1);
    
    \node[vertex] (vv1) at (36:4.08) {};
    \node[vertex] (vv2) at (108:4.08) {};
    \node[vertex] (vv3) at (180:4.08) {};
    \node[vertex] (vv4) at (252:4.08) {};
    \node[vertex] (vv5) at (324:4.08) {};

    \node[vertex] (v1) at (0:4) {};
    \node[vertex] (v2) at (72:4) {};
    \node[vertex] (v3) at (144:4) {};
    \node[vertex] (v4) at (216:4) {};
    \node[vertex] (v5) at (288:4) {};
    \draw (v1)--(vv1)--(v2)--(vv2)--(v3)--(vv3)--(v4)--(vv4)--(v5)--(vv5)--(v1);
    \draw (v1)--(v2)--(v3)--(v4)--(v5)--(v1);
     \draw (vv1)--(vv2)--(vv3)--(vv4)--(vv5)--(vv1);
    \draw (w2)--(v1)--(w5)--(v4)--(w3)--(v2)--(w1)--(v5)--(w4)--(v3)--(w2);
\end{tikzpicture}

%% file: figures/CanonicalStripProof.tex
\begin{tikzpicture}[scale=0.6]
\tikzstyle{vertex}=[rectangle, draw, fill=black,
                        inner sep=0pt, minimum width=0pt, minimum height=5pt]
\tikzstyle{vertex2}=[circle, draw, fill=black,
                        inner sep=0pt, minimum width=4pt]

    \node[] (G1) at (-2,5) { \large $G_1$ };
    \node[] (G2) at (-2,1) { \large$G_2$ };

    \node[vertex] (a1) at (0,1) {};
    \node[vertex] (a2) at (3.75,1) {};
    \draw (a1)--(a2);
    \node[vertex2] (a) [label=below:$a$] at (3,1){};

    \node[vertex] (d1) at (4.25,0.5) {};
    \node[vertex] (d2) at (9.5,0.5) {};
    \draw (d1)--(d2);

    \node[vertex] (b1) at (13.5,1) {};
    \node[vertex] (b2) at (21,1) {};
    \draw (b1)--(b2);
    \node[vertex2] (b) [label=below:$b$] at (15,1){};

    \node[vertex] (c1) at (12.5,1.5) {};
    \node[vertex] (c2) at (16,1.5) {};
    \draw (c1)--(c2);

    \node[] (dots) at (11,1) { $\ldots$ };

    \node[] (X2) at (1.5, -0.5) {\large $X_2$};
    \node[] (Y2) at (18, -0.5) {\large $Y_2$};

    \draw[color=teal, double] (-0.5,0) rectangle (3.5,2);
    \draw[color=teal, double] (14.5,0) rectangle (21.5,2);
    \draw[] (4,0) rectangle (14,2);

    \node[] (X1) at (2.5, 6.5) {\large $X_1$};
    \node[] (Y1) at (18.5, 6.5) {\large $Y_1$};

    \draw[color=teal, double] (-0.5,4) rectangle (5,6);
    \draw[color=teal, double] (16,4) rectangle (21.5,6);
    \draw[] (5.5,4) rectangle (15.5,6);



    \foreach \x in {0,0.5,1,1.5,2,2.5,3}{
        \draw[color=teal] (\x, 2) to (\x*1.5,4);
        }

    \foreach \x in {0,0.5,1,1.5,2,2.5,3, 3.5, 4}{
        \draw[color=teal] (15+\x*1.5, 2) to (16.5+\x*1.13,4);
        }

\end{tikzpicture}

%% file: figures/antihat_strip3.tex
\begin{tikzpicture}[scale=0.7]
\tikzstyle{vertex}=[circle, draw, fill=black,
                        inner sep=0pt, minimum width=4pt]

    \node[] (A) at (1, 9.5) {\large $A$};
    \node[vertex] (a1) [label=below:$a_1$] at (1,8) {};
    \node[vertex] (a2) [label=below:$a_2$] at (1,6.5) {};
    \node[vertex] (ak) [label=below:$a_k$] at (1,3.5) {};
    \node[] (dots) at (1,5) {$\vdots$};

    \node[] (B) at (9, 9.5) {\large $B$};
    \node[vertex] (b1) [label=below:$b_1$] at (9,8) {};
    \node[vertex] (b2) [label=below:$b_2$] at (9,6.5) {};
    \node[vertex] (bk) [label=below:$b_k$] at (9,3.5) {};
    \node[] (dots) at (9,5) {$\vdots$};
    
    \node[] (C) at (5, 9.5) {\large $C$};
    \node[vertex] (c1) [label=below:$c_1$] at (5,8) {};
    \node[vertex] (c2) [label=below:$c_2$] at (5,6.5) {};
    \node[vertex] (ck) [label=below:$c_k$] at (5,3.5) {};
    \node[] (dots) at (5,5) {$\vdots$};

    \foreach \x in {1,2,k}{
        \draw[] (a\x) to [bend left = 15] (b\x);
        }

    \foreach \y in {a,b,c}{
        \draw[gray] (\y 1)-- (\y 2); 
        \draw[gray] (\y 2) to [bend left = 20] (\y k) to [bend right] (\y1);
        }
        
    \foreach \w in {a,b}{
        \draw (\w k)--(c1)--(\w2);
        \draw (\w k)--(c2)--(\w1);
        \draw (\w1)--(ck)--(\w2);
    }

    \draw[color=teal, double] (0,2.5) rectangle (2,8.5);
    \draw (4,2.5) rectangle (6,8.5);
    \draw[color=teal, double] (8,2.5) rectangle (10,8.5);
\end{tikzpicture}

%% file: figures/strange_strip.tex
\begin{tikzpicture}[scale=0.7]
\tikzstyle{vertex}=[circle, draw, fill=black,
                        inner sep=0pt, minimum width=4pt]
    \draw[color=teal, double] (0,5.5) rectangle (2,8.5);
    \draw (4, 5.5) rectangle (6,8.5);
    \draw[color=teal, double] (8, 4) rectangle (10,8.5);
    
    \node[] (A) at (1,9.5) {\large $A$};
    \node[vertex] (a1) [label=below:$a_1$] at (1,8) {};
    \node[vertex] (a2) [label=below:$a_2$] at (1,6.5) {};
    
    \node[] (B) at (9,9.5) {\large $B$};
    \node[vertex] (b1) [label=below:$b_1$] at (9,8) {};
    \node[vertex] (b2) [label=below:$b_2$] at (9,6.5) {};
    \node[vertex] (b3) [label=below:$b_3$] at (9,5) {};
    
    \node[] (C) at (5,9.5) {\large $C$};
    \node[vertex] (c1) [label=below:$c_1$] at (5,8) {};
    \node[vertex] (c2) [label=below:$c_2$] at (5,6.5) {};

    \draw[] (a1) to [bend left=20](b1)--(c2)--(a2)--(c1)--(b2)--(c2)--(a1);
    \draw[ ] (c1)--(b3);

    \draw [lightgray] (a1)--(a2);
    \draw [ ,lightgray] (b1)--(b2)--(b3) to [bend right] (b1);
    \draw [ , lightgray] (c1)--(c2);

\end{tikzpicture}

%% file: figures/gear_strip2.tex
\begin{tikzpicture}[scale=0.65]
    \tikzstyle{vertex}=[circle, draw, fill=black,inner sep=2pt, minimum width=4pt]
    \tikzstyle{vertex1}=[rectangle, draw, fill=black, inner sep=2.5pt, minimum width=2.5pt]
    \tikzstyle{vertex2}=[rectangle, draw, inner sep=2.5pt, minimum width=2pt]
    \tikzstyle{vertex3}=[circle, draw, inner sep=2pt]
    \tikzstyle{vertex4}=[circle, draw, inner sep=0pt, fill=white]
    \draw[color=teal, double] (-1,-1) rectangle (1,4);
    \draw[color=teal, double] (9,-1) rectangle (11,4);

    \node[vertex1] (v1) [label=above: $v_1$] at (0,3) {};
    \node[vertex2] (v2) [label=below: $v_2$] at (0,0) {};
    \node[vertex3] (v3) [label=above: $v_3$] at (5,0) {};
    \node[vertex] (v4) [label=below: $v_4$] at (10,0) {};
    \node[vertex3] (v5) [label=above: $v_5$] at (10,3) {};
    \node[vertex2] (v6) [label=below: $v_6$] at (5, 3) {};
    \node[vertex] (v7) [label=above: $v_7$] at (2, 1.5) {};
    \node[vertex1] (v8) [label=above: $v_8$] at (8, 1.5) {};
    

    \node[vertex4] (v9) [label=below: $v_9$] at (5,-1.5) {$\star$};
    \node[vertex4] (v10) [label=above: $v_{10}$] at (5,4.5) {$\star$};

    \draw (v1)--(v2)--(v3)--(v4)--(v5)--(v6)--(v1);
    
    \draw (v6)--(v7)--(v1);
    \draw (v2)--(v7)--(v3);

    \draw (v3)--(v8)--(v4);
    \draw (v5)--(v8)--(v6);
    \draw (v7)--(v8);

    \draw (v3) -- (v9) --(v4);
    \draw (v6)to [bend left](v9) to [bend left](v1);
    \draw (v7)--(v9)--(v8);

    \draw (v2)to [bend left](v10)to [bend left](v3);
    \draw (v5)--(v10)--(v6);
    \draw (v7)--(v10)--(v8);
    
\end{tikzpicture}

%% file: figures/icosahedron.tex
\begin{tikzpicture}[scale=0.3, rotate=30]
    \tikzstyle{vertex1}=[rectangle, draw=teal, fill=teal, inner sep=2.5pt, minimum width=2pt]
    \tikzstyle{vertex2}=[rectangle, draw=teal, fill=white, inner sep=2.5pt, minimum width=2pt, thick]
    \tikzstyle{vertex3}=[circle, draw=orange, fill=orange, inner sep=2pt, minimum width=2pt]
    \tikzstyle{vertex4}=[circle, draw=orange, fill=white, inner sep=2pt, minimum width=2pt, thick]
    
    \node[vertex1] (w1) at (0:1) {};
    \node[vertex2] (w2) at (120:1) {};
    \node[vertex3] (w3) at (240:1) {};
    
    \node[vertex4] (v1) at (0:2) {};
    \node[vertex3] (v2) at (60:2) {};
    \node[vertex4] (v3) at (120:2) {};
    \node[vertex1] (v4) at (180:2) {};
    \node[vertex4] (v5) at (240:2) {};
    \node[vertex2] (v6) at (300:2) {};
    
    \node[vertex2] (u1) at (60:6) {};
    \node[vertex3] (u2) at (180:6) {};
    \node[vertex1] (u3) at (300:6) {};

    \draw (u1)--(v2)--(w1)--(w3)--(w2)--(w1)--(v1)--(u1)--(v3)--(u2)--(v5)--(u3)--(v6)--(w3)--(v4)--(w2)--(v2)--(u1);
    \draw (w1)--(v6)--(v5)--(v4)--(v3)--(v2)--(v1)--(v6);
    \draw (v4)--(u2)--(u3)--(u1)--(u2);
    \draw (v5)--(w3);
    \draw (v3)--(w2);
    \draw (v1)--(u3);
\end{tikzpicture}